\documentclass[a4paper]{amsart}
\usepackage[foot]{amsaddr}

\calclayout

\usepackage[utf8]{inputenc}
\usepackage[T1]{fontenc}
\usepackage[english]{babel}

\usepackage{amsthm} 
\usepackage[ruled,vlined]{algorithm2e}
\usepackage{amsfonts}
\usepackage{amssymb} 
\usepackage{amsmath}
\usepackage{mathtools}
\usepackage{mathpartir} 
\usepackage[all,cmtip]{xy}
\usepackage{thmtools}

\usepackage[colorlinks=false,hidelinks]{hyperref}
\usepackage{cleveref}

\usepackage{enumitem}
\setenumerate{label=(\roman*),itemsep=3pt}

\declaretheorem[
name = Theorem,
]{theorem}
\declaretheorem[
name = Corollary,
sibling = theorem
]{corollary}
\declaretheorem[
name = Lemma,
sibling = theorem
]{lemma}

\declaretheorem[
name = Proposition,
sibling = theorem
]{proposition}

\declaretheorem[
name = Claim,
numbered = no
]{claim*}

\declaretheoremstyle[%
  spaceabove=-6pt,%
  spacebelow=6pt,%
  headfont=\normalfont\itshape,%
  postheadspace=1em,%
  qed=$\blacksquare$,%
  headpunct={.}
]{mystyle}

\declaretheorem[
name = Remark,
sibling = theorem,
style=definition
]{remark}

\declaretheorem[
name = Definition,
sibling = theorem,
style=definition
]{definition}

\declaretheorem[
name = Question,
sibling = theorem,
style=definition
]{question}

\declaretheorem[
name = Convention,
sibling = theorem,
style = definition
]{convention}

\renewcommand{\DeclareMathOperator}[1]{\newcommand{#1}}

\DeclareMathOperator{\dom}{\mathrm{dom}}
\newcommand{\Pow}{\wp}

\DeclareMathOperator{\On}{\mathrm{On}}
\DeclareMathOperator{\Ord}{\mathrm{On}}

\DeclareMathOperator{\Le}{\mathbf{L}}
\DeclareMathOperator{\Ve}{\mathbf{V}}

\DeclareMathOperator{\Lo}{\mathcal{L}}

\newcommand{\seq}[1]{( #1 )}

\newcommand{\Goedel}{\mathfrak{g}}

\usepackage{color}
\definecolor{lightgray}{gray}{0.75}

\newcommand{\vir}[1]{``#1''}

\newcommand{\set}[1]{\{ #1 \}}
\newcommand{\Set}[2]{\{ #1 \, | \, #2 \}}
\newcommand{\Seq}[2]{\langle #1 \, | \, #2 \rangle}

\newcommand{\essdom}{\mathrm{essdom}}
\newcommand{\IKP}{\mathsf{IKP}}
\newcommand{\IPC}{\mathbf{IPC}}

\newcommand{\efrac}[2]{%
  \mathchoice
    {\ooalign{%
      $\genfrac{}{}{1.2pt}0{\hphantom{#1}}{\hphantom{#2}}$\cr%
      $\color{white}\genfrac{}{}{.4pt}0{\color{black}#1}{\color{black}#2}$}}%
    {\ooalign{%
      $\genfrac{}{}{1.2pt}1{\hphantom{#1}}{\hphantom{#2}}$\cr%
      $\color{white}\genfrac{}{}{.4pt}1{\color{black}#1}{\color{black}#2}$}}%
    {\ooalign{%
      $\genfrac{}{}{1.2pt}2{\hphantom{#1}}{\hphantom{#2}}$\cr%
      $\color{white}\genfrac{}{}{.4pt}2{\color{black}#1}{\color{black}#2}$}}%
    {\ooalign{%
      $\genfrac{}{}{1.2pt}3{\hphantom{#1}}{\hphantom{#2}}$\cr%
      $\color{white}\genfrac{}{}{.4pt}3{\color{black}#1}{\color{black}#2}$}}%
}

\usepackage[
backend=biber,
style=numeric,
citestyle=numeric,
doi=true,isbn=false,url=true,eprint=false
]{biblatex}

\begin{document}
\renewcommand{\phi}{\varphi} 
    
\title[Realisability for Infinitary Intuitionistic Set Theory]{Realisability for Infinitary \\ Intuitionistic Set Theory}

\author[Merlin Carl]{Merlin Carl$^1$}
\address{$^1$Institut für mathematische, naturwissenschaftliche und technische Bildung, Abteilung für Mathematik und ihre Didaktik, Europa-Universität Flensburg, Auf dem Campus 1, 24943 Flensburg, Germany}
\email{merlin.carl@uni-flensburg.de}

\author[Lorenzo Galeotti]{Lorenzo Galeotti$^2$}
\address{$^2$Amsterdam University College, P.O. Box 94160, 1090 GD Amsterdam, The Netherlands}
\email{l.galeotti@uva.nl}

\author[Robert Passmann]{Robert Passmann$^3$}
\address{$^3$Institute for Logic, Language and Computation, Faculty of Science, University of Amsterdam, P.O. Box 94242, 1090 GE Amsterdam, The Netherlands}
\email{r.passmann@uva.nl}

\thanks{The research of the third author was supported by a doctoral scholarship of the \emph{Studienstiftung des deutschen Volkes} (German Academic Scholarship Foundation).}

\date{\today}

\begin{abstract}
    We introduce a realisability semantics for infinitary intuitionistic set theory that is based on Ordinal Turing Machines (OTMs). We show that our notion of OTM-realisability is sound with respect to certain systems of infinitary intuitionistic logic, and that all axioms of infinitary Kripke-Platek set theory are realised. Finally, we use a variant of our notion of realisability to show that the propositional admissible rules of (finitary) intuitionistic Kripke-Platek set theory are exactly the admissible rules of intuitionistic propositional logic.
\end{abstract}

\maketitle

\section{Introduction}

Realisability formalises how a statement can be \textit{effectively} or \textit{explicitly} established. For example, in order to \textit{realise} a statement of the form $\forall{x}\exists{y}\phi$ according to Kleene's realisability semantics for arithmetic \cite{Kleene1945}, one needs to come up with a uniform method of obtaining a suitable $y$ from any given $x$. Such a \textit{method} is often taken to be a Turing program. Realisability originated as a formalisation of intuitionistic semantics. Indeed, it turned out to be well-chosen in this respect: the Curry-Howard-isomorphism shows that proofs in intuitionistic arithmetic correspond---in an effective way---to programs realising the statements they prove, see, e.g, van Oosten's survey \cite{vanOosten2002} for a general introduction to realisability.

From a certain perspective, however, Turing programs are quite limited as a formalisation of the concept of \textit{effective method}. Hodges \cite{Hodges} argued that mathematicians rely (at least implicitly) on a concept of effectivity that goes far beyond Turing computability, allowing effective procedures to apply to transfinite and uncountable objects. Koepke's \cite{Koepke} Ordinal Turing Machines (OTMs) provide a natural approach for modelling transfinite effectivity.

With such a transfinite generalisation of Turing computability, it becomes a natural application to obtain notions of realisability based on models 
of transfinite computability. Such a concept was first defined and briefly studied by the first author \cite{Carl2019}. In contrast to Turing machines, for which input and output are finite strings and can thus be encoded as natural numbers, OTMs can operate on (codes for) arbitrary sets. The natural domain for the concept of transfinite realisability obtained from OTMs is thus set theory rather than (transfinite) arithmetic. 
Does OTM-realisability correspond to provability in various systems of intuitionistic set theory? This question was first tackled by the first author in the context of finitary logic in a recent note \cite{Carl2019note}, which the present paper replaces and considerably expands. 

Given the presence of well-established concepts of infinitary logics, proofs, and their intuitionistic variants, it is more natural to consider how these, rather than classical provability in finitary logic, relate to OTM-realisability. In particular, we consider the question whether there is a transfinite analogue of the Curry-Howard-isomorphism; and to this end, we show that proofs in the infinitary intuitionistic proof calculus proposed by Espíndola \cite{Espindola2018} correspond to OTM-realisations. Thus, OTM-realisability is a natural concept in the context of infinitary intuitionistic proof theory. 

Previous notions of realisability for set theory were developed by Myhill, Friedman, Beeson, McCarty, Rathjen and others for a variety of intuitionistic and constructive set theories \cite{beeson1979continuity,beeson2012foundations,friedman1973some,mccarty1986realizability,myhill1973some,Rathjen2008,Rathjen2006,Rathjen2005} making explicit or implicit use of partial combinatory algebras (pca's). This notion differs from our realisability in the following two senses. First, the treatment of the existential quantifier is different: while pca-realisability usually requires witnesses for existential quantifiers to be computed uniformly, OTM-realisability allows to take parameters, e.g. witnesses selected by universal quantifiers, into account. Second, OTM-realisability allows to treat infinitary languages of arbitrary size while pca-realisability is---in its full generality---restricted to the countable infinite (as, for example, the natural numbers form a pca if one fixes an appropriate coding and application of partially recursive functions). Of course, it may be possible to circumvent these restrictions of pca-realisability by requiring additional structure on the pca's involved.

This paper is organised as follows. After introducing some necessary preliminaries and codings in \Cref{Section: Preliminaries}, we define a notion of OTM-realisability for set theory in \Cref{Section: A Notion of Transfinite Realisability}. \Cref{Section: Soundness} provides our main results on soundness on both the level of infinitary intuitionistic first-order logic as well as the level of set theory. We put our machinery to use in \Cref{Section: Proof-Theoretic Application} and prove a result about the admissible rules of finitary intuitionistic Kripke-Platek set theory, answering a question by Iemhoff and the third author (see \cite{IemhoffPassmann2020}).

\section{Preliminaries}
\label{Section: Preliminaries}

Throughout this paper, our meta-theory is $\mathsf{ZFC}$, i.e., Zermelo Fraenkel set theory with the axiom of choice. The main reason for doing so is that our arguments heavily rely on the theory of Ordinal Turing Machines (OTMs) which, so far, has only been developed within classical set theories.\footnote{One of our referees pointed out that it would be desirable to indicate which of our arguments work constructively and which depend on the use of classical logic and set theory. This, however, would require a constructive reworking the theory of transfinite computability, which seems to be a considerable research project in itself, albeit a potentially quite worthwhile one.}

\subsection{Infinitary Intuitionistic Logic}
\label{Section: Infinitary Intuitionistic Logic}

We will denote the class of ordinals by $\On$, the class of binary sequences of ordinal length by $2^{<\On}$, and the class of sets of ordinal numbers by $\wp(\On)$. We fix a class of variables $x_i$ for each $i \in \Ord$. Given an ordinal $\alpha$, a \emph{context} of length $\alpha$ is a sequence $\mathbf{x} = \Seq{x_{i_j}}{j < \alpha}$ of variables. In this paper, we will use boldface letters, $\mathbf{x}, \mathbf{y}, \mathbf{z}, \dots$, to denote contexts and light-face letters, $x_i, y_i, z_i, \dots$, to denote the $i$-th variable symbol of $\mathbf{x}$, $\mathbf{y}$, and $\mathbf{z}$, respectively. We will denote the \emph{length} of a context $\mathbf{x}$ by $\ell(\mathbf{x})$.  The formulas of the infinitary language $\Lo^\in_{\infty, \infty}$ of set theory are defined as the smallest class of formulas closed under the following rules:
\begin{enumerate}
    \item $\bot$ is a formula;
    \item $x_i \in x_j$ is a formula for any variables $x_i$ and $x_j$;
    \item $x_i = x_j$ is a formula for any variables $x_i$ and $x_j$;
    \item if $\phi$ and $\psi$ are formulas, then $\phi \rightarrow \psi$ is a formula;
    \item if $\phi_\alpha$ is a formula for every $\alpha < \beta$, then $\bigvee_{\alpha < \beta} \phi_\alpha$ is a formula;
    \item if $\phi_\alpha$ is a formula for every $\alpha < \beta$, then $\bigwedge_{\alpha < \beta} \phi_\alpha$ is a formula;
    \item if $\mathbf{x}$ is a context of length $\alpha$, then $\exists^\alpha \mathbf{x} \phi$ is a formula; and,
    \item if $\mathbf{x}$ is a context of length $\alpha$, then $\forall^\alpha \mathbf{x} \phi$ is a formula.
\end{enumerate}

By this definition, our language allows set-sized disjunctions and conjunctions as well as quantification over set-many variables at once. However, infinite alternating sequences of existential and universal quantifiers are excluded from this definition. 

\begin{remark}
While we only consider the logic $\Lo_{\infty,\infty}$ in this paper, our results can be easily adapted to logics $\Lo_{\kappa,\kappa}$ for regular cardinals $\kappa$. In particular, notions of realisability for these logics can be obtained through $\kappa$-Turing machines ($\kappa$-TMs; see \cite{Carl2019}).
\end{remark}

Whenever it is clear from the context, we will omit the superscripts from the quantifiers and write $\exists$ and $\forall$ instead of $\exists^{\alpha}$ and $\forall^\alpha$, respectively. It will often be useful to identify a variable $x$ with the context $\mathbf{x} = \langle x \rangle$ whose unique element is $x$. In such situations, we will write ``$\exists x \phi$'' for ``$\exists \mathbf{x} \phi$'' and ``$\forall x \phi$'' for ``$\forall \mathbf{x} \phi$''. A variable $x_i$ is called a \emph{free variable of a formula $\phi$} whenever $x_i$ appears in $\phi$ but is not in the scope of a quantification over a context containing $x_i$. As usual, a formula without free variables is called a \emph{sentence}. We say that $\mathbf{x}$ is \emph{a context of the formula $\phi$} if all free variables of $\phi$ appear in $\mathbf{x}$. As usual, we will write $\phi(\mathbf{x})$ if $\phi$ is a formula and $\mathbf{x}$ a context of $\varphi$. Similarly, given two contexts $\mathbf{x}$ and $\mathbf{y}$ with $x_{j} \neq y_{j'}$ for all $j < \ell(\mathbf{x})$ and $j' < \ell(\mathbf{y})$, we will write $\varphi(\mathbf{x},\mathbf{y})$ if the sequence obtained by concatenating $\mathbf{x}$ and $\mathbf{y}$ is a context for $\varphi$.

Let $\mu$ be an ordinal, $\mathbf{x}$ be a context of length $\mu$, and $y$ be a variable. Then ``$\mathbf{x} \in y$'' is an abbreviation of the following infinitary formula:
\begin{multline*}
    \exists f \exists d \exists^{\mu}\mathbf{z} [\text{\vir{$f$ is a function whose domain is the ordinal $d$}} \ \land \\ 
    \left(\bigwedge_{j'<j<\mu}(z_{j'}\in z_{j} \land z_{j} \in d \land f(z_{j})=x_{j}) \ \land \ 
    \forall x (x\in d\rightarrow \bigvee_{j<\mu}z_{j}=x)\right) \land f \in y ].
\end{multline*}
Intuitively, ``$\mathbf{x} \in y$'' expresses that there is a function $f$ such that $f = \mathbf{x}$ and $f$ is contained in the set $y$. We will later see in Lemma \ref{Lemma:sequenceBound} that ``$\mathbf{x} \in y$'' is interpreted with the intended set-theoretical meaning, namely that the set $y$ contains the sequence $\mathbf{x}$. 

As in the case of finitary realisability, bounded quantification will play a crucial role for transfinite realisability. For this reason, we extend the classical abbreviations as follows: given a formula $\phi$ and an ordinal $\alpha \geq \omega$ we introduce the \emph{bounded quantifiers} as abbreviations, namely,  
    $$
        \forall^{\alpha} \mathbf{x}\in y \ \phi \text{ for } \forall^\alpha \mathbf{x} (\mathbf{x}\in y \rightarrow \phi),
    $$
and 
    $$
        \exists^{\alpha} \mathbf{x}\in y \ \phi \text{ for } \exists^{\alpha} \mathbf{x} ( \mathbf{x}\in y \land \phi).
    $$
The bounded quantifiers for $\alpha < \omega$ are defined as usual. The class of $\Delta^{\omega}_0$-formulas consists of those formulas that have no infinitary quantifiers and whose quantifiers are bounded\footnote{We note that $\Delta^{\omega}_0=\Delta_0$ and $\Sigma^{\omega}_1=\Sigma_1$ where $\Delta_0$ and $\Sigma_1$ are the usual classes of formulas in the Levy hierarchy.}. Similarly, a formula belongs to the class of $\Sigma^{\omega}_1$-formulas if it is of the form $\exists x \psi$ for some $\Delta^{\omega}_0$-formula $\psi$. We extend this definition to formulas with infinitary quantifiers as follows. An infinitary formula is a $\Delta^{\infty}_0$-formula if all the quantifiers appearing in the formula are bounded. Furthermore, the class of $\Sigma^{\infty}_1$-formulas consists of the formulas of the form $\exists^\alpha \mathbf{x}\psi$ for some $\Delta^{\infty}_0$-formula $\psi$ and an ordinal $\alpha$. 

\begin{remark}
    Note that the previous definition of infinitary $\Delta^{\infty}_0$-formulas requires the bounding set to contain the sequence $\langle x_{j} \mid j < \alpha \rangle$ rather than just each individual element of that sequence. This is necessary to lift the usual absoluteness results to the infinitary case. Indeed, the following alternative definition and straightforward generalisation of the standard definition of bounded formulas does not provide absoluteness. If we define
    $$
        \underline{\forall}^{\alpha} \mathbf{x}\in y \phi \text{ as } \forall^\alpha \mathbf{x} (\bigwedge_{j< \alpha } x_{j}\in y \rightarrow \phi),
    $$
    and 
    $$
        \underline{\exists}^{\alpha} \mathbf{x}\in y \phi \text{ as } \exists^{\alpha} \mathbf{x} (\bigwedge_{j<\alpha} x_{j}\in y \land \phi),
    $$ 
    then it 
    the following formula
    $$
        \phi(y) := \underline{\exists}^{\omega} \mathbf{x} \in \{0,1\} \ \mathbf{x} \notin y
    $$
    is not absolute: it is easy to see that if $\Pow^{\Ve}(\omega)\neq\Pow^{\Le}(\omega)$, then  $\Le \models \lnot\phi(\Le_{\omega_1^\Le})$ but $\phi(\Le_{\omega_1^\Le})$ is true.

\end{remark}

\subsection{Ordinal Turing Machines}

Ordinal Turing machines (OTMs, for short) were introduced by Koepke \cite{Koepke} as a transfinite generalisation of Turing machines and will be the main ingredient for our definition of infinitary realisability. We will only give a basic intuition for this model of transfinite computability and refer to the relevant literature (e.g. \cite{Koepke} or in \cite[Section 2.5.6]{Carl2019}) for a full introduction to OTMs. 

An \textit{Ordinal Turing machine (OTM)} has the following tapes of unrestricted transfinite length: finitely many tapes for the input, finitely many scratch tapes, and one tape for the output. Ordinal Turing machines run classical Turing machine programs and behave exactly like standard Turing machines at successor stages of a computation. At limit stages, the content of the tapes is computed by taking the point-wise inferior limit, the position of the head is set to the inferior limit of the head positions at previous stages, and the state of the machine is computed using the inferior limit of the states at previous stages (see \cite[Section 2.5.6]{Carl2019} for more details). 

Ordinal Turing machines are a very well-behaved model of transfinite computability and many results from classical computability theory can be generalised to OTMs (e.g., \cite{Koepke,Dawson,Koepke09,Rin, Carl2017,Carl2018,Seyfferth}). For this reason, we will describe OTM-programs using high-level pseudo algorithms as usually done with Turing machines. We fix a computable coding of Turing machine programs as natural numbers. In the rest of the paper, we will identify a program with its natural number code. 

In this paper, we consider machines that run programs with parameters. For this purpose, we fix one of the input tapes as the parameter tape. A parameter is a binary sequence of ordinal length, i.e., an element of $2^{< \Ord}$, which is written on the parameter tape before the execution of the program begins. Using classical techniques (see below), we can code sets of parameters in a single sequence $p \in 2^{<\Ord}$. Therefore, we can say that an OTM executes a program $c$ \emph{with parameter $P \subset 2^{< \Ord}$} if a code for $P$ is written on the parameter tape before the program is executed.\footnote{We assume here a fixed order in which the parameter appears in the coding.} Given a program $c$ and a parameter $P \subset 2^{< \Ord}$, and two sequences $a$ and $b$ in $2^{<\Ord}$ we will write $c_P(a)=b$ if an OTM which executes $c$ with parameter $P$ and input $a$, outputs $b$. 

Given a class function $f: 2^{<\Ord} \rightarrow 2^{<\Ord}$ we say that \emph{a program $c$ computes $f$ with parameter $P \subset 2^{<\Ord}$} if for every binary sequence $b$ we have $c_P(b)=f(b)$. 

\subsection{Coding}

The objects of our notion of realisability will be arbitrary \textit{sets} and not just ordinal numbers. Therefore, we need to be able to perform computations on sets. As OTMs work on ordinal length binary sequences, we have to code arbitrary sets as sequences in $2^{<\Ord}$. However, working directly on binary sequences of ordinal length is cumbersome. We therefore use a concatenation of two codings given by the following injective functions:
\begin{enumerate}
    \item A \textit{low-level coding} 
    $$2^{<\Ord} \to \Ord \cup \wp(\Ord) \cup (\Ord \times \wp(\Ord)) \cup (\Ord \times \wp(\Ord))^{<\Ord}$$
    allowing us to encode ordinals, sets of ordinals, pairs of sets of ordinals and ordinals, and sequences of such pairs as binary sequences of ordinal length.
    \item A \textit{high-level coding} $$\mathtt{C}:  (\On \times \wp(\Ord)) \to \Ve$$ encoding arbitrary sets as a pair of an ordinal and a set of ordinals. 
\end{enumerate}
By concatenating both injective functions, we obtain a coding of arbitrary sets as binary sequences of ordinal length. The reader familiar with OTMs and codings may wish to skip forward to Section \ref{Section: A Notion of Transfinite Realisability} and refer back to this section if necessary.

\subsubsection{Low-level coding}
While not difficult, the technical details of the low-level coding are complicated and cumbersome (just like in the case of codings for ordinary Turing machines). We will now introduce a coding of sets of ordinals as binary sequences of ordinal length and prove a few basic properties of this coding.

\begin{remark}
    Note that sets of ordinals can in principle also be coded using characteristic functions. Unfortunately, since the largest ordinal in the set is not computable from this coding, bounded searches are not computable with this simple coding and we would not be able to compute basic operations over sets, such as computing the image of a set under a computable function.
\end{remark}

Given a binary sequence $b\in 2^{<\On}$ and a set $ X\in \wp(\On)$ of ordinals we say that $b$ \emph{codes} $X$ if for $\beta = \sup\{2\alpha+2 \mid \alpha \in X\}$ we have $b(\beta+1) = b(\beta+2) = 1$, for every $2\alpha \leq \beta$ we have that $b(2\alpha)=0$, and for $2\alpha+1 < \beta$ we have $\alpha\in X$ if and only if $b(2\alpha+1)=1$. Intuitively, a binary sequence encodes a set of ordinals $X$ if it is of the form 
    $$
        0 i_0 0 i_1 0 \ldots 0 i_\alpha 0 \ldots 1 1 \ldots
    $$ 
where $i_\alpha = 1$ if and only if $\alpha \in X$. Note that the final sequence $11$ marks the end of the code of the set, i.e, none of the further bits matter in determining the set encoded by the sequence.\footnote{For example, $0$ is encoded by any sequence starting with $011$, $1 = \{ 0 \}$ is encoded by any sequence starting with $01011$, and $\omega$ is encoded by any sequence starting with the sequence $\underbrace{0101010101\ldots}_{\text{length } \omega}011$.} 

In Section \ref{Section:Codings of Sets}, we code sets as pairs $\langle \alpha, X \rangle$ consisting of an ordinal $\alpha$ and a set $X$ of ordinals including $\alpha$. For this reason, we will now extend the previous coding to pairs in $\On \times \wp(\On)$. The idea is to start with a code of $X$ as described before and encode $\alpha$ in the final part of the binary sequence. A sequence $b \in 2^{<\On}$ \emph{encodes} $\langle \alpha, X\rangle$ if it encodes $X$ and for all $\beta+2<\gamma<\beta+3+\alpha<\eta$ we have $b(\gamma)=b(\eta)=0$, and $b(\beta+3+\alpha)=1$. Intuitively, a pair $\langle \alpha, X \rangle\in \On\times \wp(\On)$ is encoded by a sequence of the form $$\underbrace{0i_00i_10\ldots0i_\gamma 0\ldots 11}_{\text{Code of $X$ of length $\beta+2$}}\underbrace{0000\ldots }_{\alpha}1000\ldots$$ where as before $i_\gamma=1$ if and only if $\gamma\in X$.

We take $\chi$ to be the class function that associates to every sequence encoding a pair $\langle \alpha,X\rangle$ the corresponding pair. It is easy to see that $\chi$ is bijective. 

We will say that a program $c$ computes a class function $F:\On\times \wp(\On)\rightarrow \On\times\wp(\On)$ with parameters $P$ if an OTM executing $c$ with parameter $P$, and a sequence encoding a pair $\langle \alpha,X \rangle \in \On\times \wp(\On)$ as input, returns a sequence that encodes $F(\langle \alpha,X\rangle)$. Moreover, we will say that $c$ computes the class function $F: \wp(\On) \rightarrow \wp(\On)$ with parameter $P$ if for every set $X \in \wp(\On)$, an OTM executing $c$ with parameter $P$, and a sequence that encodes $\langle 0,X\rangle$ as input, returns a sequence that encodes $\langle 0, F(X)\rangle$.

The previous coding can be easily extended to sequences in $(\On\times \wp(\On))^{<\On}$. We encode a sequence $\langle \langle \alpha_\beta,X_\beta \rangle \mid \beta <\eta\rangle$ using the sequence obtained by concatenating the encodings $\chi^{-1}(\langle \alpha_\beta,X_\beta \rangle)$ in the order they appear in $\langle \langle \alpha_\beta,X_\beta \rangle \mid \beta <\eta\rangle$ followed by a sequence of four $1$s to mark the end of the code of the sequence. Therefore, a sequence $\langle \langle \alpha_\beta,X_\beta \rangle \mid \beta <\eta\rangle$ is coded as follows:
$$
    \underbrace{0i^{0}_00i^{0}_10\ldots0i^0_{\gamma_0}0\ldots 11}_{\text{Code of $X_{0}$}}
    \underbrace{0000\ldots }_{\alpha_{0}}1
    \underbrace{0i^{1}_00i^{1}_10\ldots0i^1_{\gamma_1}0\ldots 11}_{\text{Code of $X_{1}$}}
    \underbrace{0000\ldots }_{\alpha_{1}}1 \ldots 1111
$$
where, as before, $i^\beta_\gamma=1$ if and only if $\gamma\in X_\beta$ for all $\beta<\eta$. As mentioned above, this coding induces a notion of computability over $(\On\times \wp(\On))^{<\On}$.

\begin{lemma}\label{Lemma:BasicORDCodingAlgo}
    Let $X$ be a set of ordinals. Then the following are OTM-computable:
    \begin{enumerate}
        \item \label{BasicORDCodingAlgo:1} the function that, given codes for an ordinal $\alpha$ and a sequence $X$, returns $1$ if $\alpha \in X$, and $0$ otherwise,
        \item \label{BasicORDCodingAlgo:2} the function that, given codes for an OTM-computable function $f:\On\rightarrow \On$  and $X$, returns a code for the image $\{f(\alpha)\mid \alpha\in X\}$ of $X$ under $f$,
        \item \label{BasicORDCodingAlgo:3} the function that, givens codes for $\langle \langle \alpha_\gamma,X_\gamma \rangle \mid \gamma <\eta\rangle$ and an ordinal $\beta<\eta$, returns a code for $\langle \alpha_\beta,X_\beta \rangle$,
        \item \label{BasicORDCodingAlgo:4} the function that, given codes for $\langle \langle \alpha_\gamma,X_\gamma \rangle \mid \gamma <\eta\rangle$ and an ordinal $\beta<\eta$, returns a code for the list obtained by removing the $\beta$-th element from $\langle \langle \alpha_\gamma,X_\gamma \rangle \mid \gamma <\eta\rangle$,
        \item \label{BasicORDCodingAlgo:5} the function that, given codes for $\langle \langle \alpha_\gamma,X_\gamma \rangle \mid \gamma <\eta\rangle$ and  $\langle \alpha,X\rangle$, returns a code of the list $\langle \langle \alpha_\gamma,X_\gamma \rangle \mid \gamma<\eta+1\rangle$ where $\langle \alpha_\eta,X_\eta \rangle=\langle \alpha,X\rangle$,
        \item \label{BasicORDCodingAlgo:6} bounded searches through sets of ordinals.
    \end{enumerate}
\end{lemma}
\begin{proof}
    Given a code of a set of ordinals $X$ as a binary sequence, and given some ordinal $\alpha$ such that $\alpha<\sup\{2\beta+2\mid \beta \in X\}$, an OTM can stop at the position of the tape which contains the $i_\alpha$ bit of the code of $X$. In what follows, we will refer to this the cell of the tape as the \emph{position} of $i_\alpha$ on the tape. This can be computed by the program that moves the head of the tape left increasing a counter $\gamma$ each time that the head is moved to a cell with an even index. Once  $\gamma$ reaches $\alpha$, the machine moves to the next position of the tape and stops. 
    
    For \ref{BasicORDCodingAlgo:1}, consider the program that goes through the code of $X$ until it reaches the cell containing $i_\alpha$, copies the content of that cell to the output tape and then stops.
    
    For \ref{BasicORDCodingAlgo:2}, consider the program that runs through the representation of $X$ and for each $\alpha \in X$, first computes $f(\alpha)$, then writes a $1$ in the position of $i_{f(\alpha)}$ on the output tape, and saves the index of the first cell of the output tape after which the output tape was not modified in an auxiliary counter $\gamma$. Once the program sees the sequence $011$ in the representation of $X$, it moves the head of the output tape to position $\gamma$, writes $011$ and stops. 
    
    For \ref{BasicORDCodingAlgo:3}, the program goes trough the code of $\langle \langle \alpha_\gamma,X_\gamma \rangle \mid \gamma <\eta\rangle$ while increasing a counter $\alpha$ as follows: each time that the program sees the sequence $011$, it looks for the first $1$ after it and then increases $\alpha$ by $1$. As soon as $\alpha = \beta$, the program copies the input tape from the current position until the first occurence of the sequence $011$, then it copies all $0$s following this and stops as soon as it reaches a cell with a $1$. 
    
    For \ref{BasicORDCodingAlgo:4}, the program goes trough the code of $\langle \langle \alpha_\gamma,X_\gamma \rangle \mid \gamma <\eta\rangle$ while increasing a counter $\alpha$ as for \ref{BasicORDCodingAlgo:3} and copying the input tape on the output tape. Once $\beta=\alpha$ the program continues to go through the representation of $X$ without copying it to the output tape the until it sees the first occurrence of the sequence $011$ and the next $1$ after that. Then the program continues copying the input tape to the output tape until it sees the sequence $1111$.
    
    Note that \ref{BasicORDCodingAlgo:5} can be trivially computed by an algorithm that copies the representation of $\langle \langle \alpha_\gamma,X_\gamma \rangle \mid \gamma <\eta\rangle$ on the output tape except for the sequence $1111$. Then the program copies the representation of $\langle \alpha,X\rangle$ on the output tape and writes the sequence $1111$ when done. 
    
    Finally, \ref{BasicORDCodingAlgo:6} can be computed by combining the previous items. Note, in particular, that our choice of coding allows us an OTM to recognise when it reaches the end of a code.
\end{proof}

By Lemma \ref{Lemma:BasicORDCodingAlgo} we are now justified in treating tapes as queues. We will make plenty of use of this convention. Moreover, we will use the expressions \vir{bounded search}, \vir{bounded search through a set}, and \vir{search through a set} interchangeably. If $c$ is a code for an OTM-program, $P$ a parameter-set, and $X$ and $Y$ are sets of ordinals, we will write $c_P(X) = Y$ to mean that an OTM-computation of the program $c$ with parameters $P$ and input $\chi^{-1}(X)$ halts with $\chi^{-1}(Y)$ on the output tape. We write $c_P(X){\uparrow}$ if this computation does not halt. 

By using the coding defined in the previous section we can extend the notion of computability to $\Ve$. Given two sets $X$ and $Y$, we write $c_P(X) = Y$ when, given a code $a$ for $X$, $c_P(a)$ computes a code of $Y$. Note that the class relation $c_P$ on $\Ve$ is in general \emph{not} a function but a multi-valued function---indeed, the set coded by the output of $c_P$ could depend on the specific code of $X$. Figure \ref{Fig:codings} shows how the codings interact with each other.

\begin{figure}[h] 
    \centerline{
    \xymatrix{
        \Ve \ar@<-.5ex>@{-->}_{}[d] \ar@<+.5ex>@{-->}^{R}[d]  \ar@{<-}^{a \, \mapsto \, d_a(\rho_a)}[rr]&&\On\times \wp(\On) \ar@{-->}^{\text{F}}[d] \ar@{<-}^{\chi}[rr]&& 2^{<\On}  \ar@{-->}^{\text{f}}[d] \\
        \Ve\ar@{<-}^{a \, \mapsto \, d_a(\rho_a)}[rr]&&\On \times \wp(\On) \ar@{<-}^{\chi}[rr]&& 2^{<\On}
    }}
    \caption{An illustration of the stratification of codings used in this paper. The double arrow indicates the fact that $R := \{\langle X,Y \rangle\mid c_P(X)=Y\}$ where $c$ is the program computing $f$ with parameter-set $P$ could be a multi-valued function. \label{Fig:codings}}
\end{figure}




Departing from this low-level coding, we need to encode arbitrary sets as pairs of an ordinal and a set of ordinals to use them as the objects of our investigation. 

\begin{convention}
    \label{Convention:low-level}
    We will follow the usual simplifying convention when working with OTMs (and, in fact, any other kind of machine), and confuse objects with their codes. For example, we might say that `an OTM takes an ordinal $\alpha$ as input' when we mean, of course, that it `takes a (low-level) code for an ordinal $\alpha$ as input,' and similar for sets of ordinals, sequences, etc.
    It will always be clear from the context when we mean low-level codes because OTMs operate directly only on binary sequences (and not on ordinals or arbitrary sets).
\end{convention}

\subsubsection{High-level coding} 
\label{Section:Codings of Sets}

We will now define a \textit{high-level coding} for arbitrary sets. Recall that the Gödel class function $\Goedel$ extends the usual Cantor pairing function on natural numbers and maps pairs of ordinals to ordinals, see, e.g., \cite[p.31]{Jech}. Given a set $X$ we will denote its transitive closure  by $\mathrm{tc}(X)$ (see, e.g., \cite[p.64]{Jech}).

\begin{definition} 
    \label{Def:pre-code}
    Let $X$ be a set. A \emph{pre-code for $X$} is a set $C$ of ordinals such that there is a bijection $d_C: \alpha \to \mathrm{tc}(\{X\})$ with the property that $\Goedel(\beta,\gamma) \in C$ if and only if $d_C(\beta) \in d_C(\gamma)$. We say that $C$ is a pre-code if it is a pre-code for some set $X$.
\end{definition}

\begin{definition}
    The \textit{high-level coding $\mathtt{C}$} is the partial surjective class function $\Ord \times \wp(\Ord) \to \Ve$ given by $\mathtt{C}(\langle \rho,C \rangle) = d_C(\rho)$. If $\mathtt{C}(\langle \rho, C \rangle) = x$, we say that $\langle \rho, C \rangle$ is a \textit{code of} $x$ or that it \textit{encodes} $x$. We will also say that $\rho$ is the \emph{representative of $x$ in $C$}.
\end{definition}

We say that a tuple $\langle \beta, C \rangle$ is a code if it is a code of some set $x$. We say that a code $c$ is based on the pre-code $C$ for $y$ if $c = \langle \beta, C \rangle$ for some ordinal $\beta$. If $c$ is a code for $x$ based on a pre-code $C$ for $y$, we will write $\rho_c$ for the representative of $x$ in $C$, and $d_c$ for the bijection $d_C$.

The pre-code on which a code is based may contain much more information than actually needed for the coding. For this reason, we introduce the following notion of \textit{essential domain} containing only the information that is strictly needed to recover the set.

\begin{definition}
    Let $c$ be a code. The \emph{essential domain $\essdom(c)$} of $c$ is the set
    $$
         \essdom(c) = \Set{\alpha\in \dom(d_c)}{ d_c(\alpha) \in \mathrm{tc}(\{ d_c(\rho_c) \})}.
    $$
\end{definition}
Intuitively speaking, the essential domain of a code $c$ of a set $x$ contains the transitive closure of $x$ as coded by $c$ but not more than that. There can be many codes for one set. For this reason, we introduce the following notion of isomorphism between codes. Crucially, these isomorphisms are only between the essential domains of a code.

\begin{definition}
    Let $a$ and $b$ be codes. A bijection $f: \essdom(a) \to \essdom(b)$ is called an \emph{isomorphism of codes} if $f(\rho_a) = \rho_b$ and $d_a(\alpha) = d_b(f(\alpha))$ for all $\alpha \in \essdom(a)$. 
\end{definition}

As usual, we will say that two codes are \emph{isomorphic} if there is an isomorphism of codes between them. 

\begin{lemma}\label{Lemma:IsoCodEq}
    Let $X$ and $Y$ be sets. If $a$ is a code of $X$, $b$ is a code of $Y$, and $f$ a code-isomorphism between them, then $X = Y$. 
\end{lemma}
\begin{proof}
    We see that $X = d_a(\rho_a) = d_b(f(\rho_a)) = d_b(\rho_b) = Y$.
\end{proof}

We will now prove some helpful properties of (high-level) codes. In doing so, we will rely on Lemma \ref{Lemma:BasicORDCodingAlgo} and Convention \ref{Convention:low-level}.

\begin{lemma}\label{Lemma:BasicAlgo}
    Let $a$ and $b$ be codes of sets.
    \begin{enumerate}
        \item \label{Lemma:BasicAlgo1} The set $\essdom(a)$ is computable, i.e., given a code for $a$, we can compute a binary sequence $b \in 2^{<\Ord}$ such that $b(\alpha)$ is $1$ if and only if $\alpha \in \essdom(a)$.
        \item \label{Lemma:BasicAlgo2} The function that given $\alpha\in \dom(d_a)$ returns the unique $\beta\in \dom(d_b)$ such that $d_a(\alpha)=d_b(\beta)$ if such an ordinal exists and returns $\dom(d_b)$ otherwise is OTM-computable.
        \item \label{Lemma:BasicAlgo3} The function that given $\alpha\in \dom(d_a)$ and $\beta\in \dom(d_b)$ returns $1$ if and only if $d_a(\alpha)\subseteq d_b(\beta)$ is OTM-computable.
        \item \label{Lemma:BasicAlgo4} There is a program that, given a code of a sequence $\langle a_\beta \mid \beta < \alpha \rangle$ of codes of sets, returns a code for the set $\{d_{a_\beta}(\rho_{a_\beta})\mid \beta<\alpha\}$. 
        \item \label{Lemma:BasicAlgo5} Let $\phi$ be a computable property of sets, i.e., there is a program $c$, possibly with parameter $P$, such that $c_P(a)=1$ if and only if $\phi[Y/x]$ for every code $a$ of a set $Y$. There is a program, possibly with parameters, that given a code of a set $X$ as input computes a code for $\{Y \in X \mid \phi[Y/x]\}$. 
    \end{enumerate}
    In particular, note that the relations $X \in Y$, $X \subseteq Y$, and $X = Y$ are computable. 
\end{lemma}
\begin{proof}
    For \ref{Lemma:BasicAlgo1}, we make use of an auxiliary tape which we will use as a queue in virtue of Lemma \ref{Lemma:BasicORDCodingAlgo}. Given a code $a = \langle \rho_a, C \rangle$, the computation of $\essdom(a)$ is implemented via a breadth-first search as follows. Enqueue $\rho_a$. Then run the following loop until there are no elements left in the queue: dequeue the first element of the queue, say ordinal $\beta$. Mark the $\beta$th position of the output tape with a $1$. Then conduct a bounded search through $C$ during which all $\alpha$ are enqueued whenever $\Goedel(\alpha,\beta) \in C$, and jump to the beginning of the loop. This procedure must eventually stop because sets are well-founded so that only set-many steps will be required to walk through the transitive closure of the set represented by $\rho_a$.
    
    For \ref{Lemma:BasicAlgo2} and \ref{Lemma:BasicAlgo3}, we provide two procedures that recursively call each other: $\mathtt{Decode}(a,b,\alpha)$ and $\mathtt{Subset}(a,b,\alpha,\beta)$, where $\mathtt{Decode}(\cdot,\cdot,\cdot)$ is the function required for \ref{Lemma:BasicAlgo2} and $\mathtt{Subset}(\cdot,\cdot,\cdot,\cdot)$ for \ref{Lemma:BasicAlgo3}. 
    
    Let's begin with $\mathtt{Decode}(a,b,\alpha)$, which takes codes $a = (\rho_a,A)$ and $b = (\rho_b,B)$ for sets as well as (a code for) an ordinal $\alpha \in \dom(d_a)$ as input. First, the program checks whether $d_a(\alpha) = \emptyset$ by conducting a bounded search looking for a $\beta \in \dom(d_a)$ such that  $\Goedel(\beta, \alpha) \in A$. If such a $\beta$ cannot be found, the program looks for the unique $\gamma$ such that $\Goedel(\beta,\gamma) \notin B$ for all $\beta \in \dom(d_B)$ which must exist as $B$ is a pre-code. The program then returns this $\gamma$ and stops. If it turns out that $d_a(\alpha)$ is non-empty, then the program goes through all the ordinals $\beta \in \dom(d_b)$ and calls $\mathtt{Subset}(a,b,\alpha,\beta)$ and $\mathtt{Subset}(b,a,\beta,\alpha)$ to check whether $d_a(\alpha) = d_b(\beta)$. If that is the case, then the program returns $\beta$. If no such $\beta$ is found, it writes $\dom(d_b)$ on the output tape.
    
    Next, consider $\mathtt{Subset}(a,b,\alpha,\beta)$, which takes as input two codes $a = \langle \rho_a,A \rangle$ and $b = \langle \rho_b,B \rangle$. This procedure should return $1$ in case $d_a(\alpha) \subseteq d_b(\beta)$ and $0$ otherwise. To do so, conduct a bounded search on $\gamma \in \dom(d_a)$ and check whether $\Goedel(\gamma,\alpha) \in A$. If so, let $\delta = \mathtt{Decode}(a,b,\gamma)$ and check whether $\Goedel(\delta,\beta) \notin B$. If this is the case, write $0$ on the output tape and stop. If the bounded search is finished without interruption, write $1$ on the output tape and stop.
    
    To finish the proof of \ref{Lemma:BasicAlgo2} and \ref{Lemma:BasicAlgo3}, note that the mutual recursive calls finish after set-many steps of computation because sets are well-founded so that, eventually, the case for $\emptyset$ will be reached. 
    
    For \ref{Lemma:BasicAlgo4}, let $\langle a_\beta \mid \beta < \alpha \rangle$ be given with $a_\beta = \langle \rho_{a_\beta}, A_\beta \rangle$. Set $\sigma_0 = \rho_{a_0}$ and recursively construct $\sigma_\beta$ for $\beta \leq \alpha$ as follows. 
    If $\beta = \gamma + 1$, then let $\sigma_\beta = \sup(\{ \sigma_\gamma + \delta \mid \delta \in A_\gamma \}) + 1$. If $\beta$ is a limit, then let $\sigma_\beta = \sup(\{ \sigma_\gamma \mid \gamma < \beta \})$. Note that all this can be straightforwardly computed by an OTM. Then define the set $A$ as follows:
    $$
        A = \left( \bigcup_{\beta < \alpha} \{ \Goedel(\sigma_\beta + \delta, \sigma_\beta + \gamma) \mid \delta, \gamma \in \Ord \text{ such that } \Goedel(\delta,\gamma) \in A_\beta \} \right) \cup \{ \Goedel(\sigma_\beta + \rho_{a_\beta},\sigma_\alpha) \mid \beta < \alpha \}.
    $$
    Intuitively, the first half obtains a concatenation of all the sets coded by $a_\beta$ and the second half makes sure that all the sets coded by $a_\beta$ are members of the new set.
    Now, we might not yet have a code: if there are $\gamma, \delta < \alpha$ such that the sets coded by $a_\gamma$ and $a_\delta$ are not disjoint, then their common elements will appear twice. For this reason, we have to remove duplicates from $A$ and obtain a pre-code $A'$. This can easily be done with a bounded search, using the $\mathtt{Subset}(\cdot,\cdot,\cdot,\cdot)$ routine, replacing duplicate occurrences with the first occurrence and removing those that are not needed. Finally, return the code $\langle \sigma_\alpha,A' \rangle$.

    Finally, \ref{Lemma:BasicAlgo5} is witnessed by the following algorithm: search through the code of $X$ and check whether $\phi[Y/x]$ holds for each $Y \in X$. If so, then the machine adds a code of $Y$ to an auxiliary queue. Once the program has checked all the elements of $X$, it uses the algorithm presented in the proof of \ref{Lemma:BasicAlgo4} to compute the desired code from the sequence saved in the queue.
\end{proof}

Given that an OTM-program has access to the particular codes of a set, it can make use of, for example, the well-ordering that is implicit in the coding. To avoid any unwarranted side-effects of this, we introduce the notion of \textit{uniform} OTM-programs. Intuitively, an OTM-program is uniform if running the program with codes for the same set(s) results in codes for the same set(s).

\begin{definition}[Uniform OTM]
    An OTM-program $c$, potentially using parameters $P$, is called \textit{uniform} if whenever two codes $\seq{\rho_0,C_0}$ and $\seq{\rho_1,C_1}$ code the same set, then $c_P(\seq{\rho_0,C_0})$ and $c_P(\seq{\rho_1,C_1})$ are (potentially different) codes of the same set.
\end{definition}

\begin{convention}
    \label{Convention:high-level}
    As with the low-level coding, we will follow the usual convention of (ordinal) computability theory and confuse objects and their codes. As in Convention \ref{Convention:low-level}, it will always be clear (or does not matter) which coding is applied when: for example, if we say that an OTM takes a set as input, then we mean that it takes as input the low-level code of a high-level code for a set. 
\end{convention}



\section{A Notion of Transfinite Realisability}
\label{Section: A Notion of Transfinite Realisability}

\subsection{OTM-Realisability}

In this section, we will define the realisability relation $\Vdash$ as a relation between potential realisers, viz., OTM programs with parameters, and formulas in the infinitary language of set theory $\Lo^\in_{\infty, \infty}$.

\begin{definition}
    A \emph{potential realiser} is a tuple $\langle c,P\rangle$ consisting of an OTM-program $c$ and a set $P$ of ordinals.
\end{definition}

\begin{remark}    
    In Kleene's classical realisability with Turing Machines for arithmetic, every element of the universe, i.e., every natural number, is computable. To make sure that our notion has the same power, we use OTMs with set parameters as potential realisers. OTMs with a single ordinal parameter can only compute every element of the universe in case that $\Ve = \Le$, and OTMs without parameters can only compute countably many elements. The seeming imbalance of finite programs with arbitrary parameters could be resolved by moving to the equivalent version of OTMs that can run programs of transfinite length studied by Lewis \cite{EthanThesis} under the supervision of Benedikt L\"owe and the second author.
\end{remark}

As potential realisers are just OTMs with parameters, we can keep our conventions and write $r(x) = y$ in case that $r = \langle c,P\rangle$ is a potential realiser such that $c_P(x) = y$. Similarly, we say that a potential realiser $r = \langle c,P\rangle$ computes a function if the OTM running program $c$ on parameters $P$ computes that function. 

To define the realisability relation, we subtly extend the infinitary language of set theory with a constant symbol for every set in the universe $\Ve$. To simplify notation, and without creating any confusion, we use the same letters to denote sets and their corresponding constant symbols.  

\begin{definition}[Substitution of Contexts]
    Let $\mathbf{x}$ be a context of length $\alpha$, $\bar X=\langle X_{i} \mid i<\alpha\rangle$ be a sequence of length $\alpha$, and $\phi$ be a formula in context $\mathbf{x}$. Then $\phi[\bar X / \mathbf{x}]$ is the formula obtained by replacing the free occurrences of the variable $x_{i}$ in $\phi$ with the (constant symbol for the) set $X_{i}$. 
\end{definition}

\begin{definition}[OTM-Realisability of $\Lo^\in_{\infty, \infty}$]
    \label{Definition:OTM-Realisability}
    The OTM-realisability relation $\Vdash$ is recursively defined as a relation between the class of potential realisers and the class of formulas in the infinitary language of set theory $\Lo^\in_{\infty, \infty}$ as follows:
    \begin{enumerate}
        \item $r \Vdash \bot$ is never true,
        \item $r \Vdash X = Y$ if and only if $r$ computes a code-isomorphism for every pair of codes for $X$ and $Y$,
        \item $r \Vdash X \in Y$ if and only if for any codes $a$ for $X$ and $b$ for $Y$, it holds that $r(a,b) = \langle \alpha, s\rangle$ such that $d_{b}(\alpha) \in Y$ and $s \Vdash X = d_b(\alpha)$,
        \item $r \Vdash \phi \rightarrow \psi$ if and only if for every $s \Vdash \phi$ we have that $r(s) \Vdash \psi$,
        \item $r \Vdash \bigvee_{\alpha < \beta} \phi_\alpha$ if and only if $r(0) = \langle \gamma,s\rangle$ such that $s \Vdash \phi_\gamma$,
        \item $r \Vdash \bigwedge_{\alpha < \beta} \phi_\alpha$ if and only if $r(\alpha) \Vdash \phi_\alpha$ for all $\alpha < \beta$, 
        \item $r \Vdash \exists \mathbf{x} \phi$ if and only if $r(0) = \langle a,s \rangle$ such that $a$ is a code for $\bar X$ and $s \Vdash \phi[\bar{X}/\mathbf{x}]$, and,
        \item $r \Vdash \forall \mathbf{x} \phi$ if and only if $r(a) \Vdash \phi[\bar{X}/\mathbf{x}]$ for every code $a$ of every sequence $\bar{X}$ of length $\ell(\mathbf{x})$.
    \end{enumerate}
    A formula $\varphi$ is \textit{realised} if there is a potential realiser $\langle c,P \rangle\Vdash \varphi$. In this situation, we will call $\langle c,P \rangle$ a \emph{realiser} of $\varphi$.
\end{definition}

OTM-programs do not necessarily give rise to functions from sets to sets because the result of a computation may depend on specific features of the code of a set. Consider, for example, a realiser $r \Vdash \forall x \exists y (x \neq \emptyset \rightarrow y \in x)$. Intuitively, the realiser $r$ must do the following: given a code for a non-empty $x$, compute the code for a set $y \in x$. A simple implementation of this would be to search through the code of $x$ until the first element of $x$ is found (if there is any) and return the code for this element. However, there are obviously codes $c_0$ and $c_1$ of  $\{ 0, 1 \}$ such that $r(c_i)$ returns a code of $i$ for $i < 2$. This shows that realisers may make use of specific features of our codes and give rise to multi-valued functions. In conclusion, it is natural to consider the following restricted notion of realisability in which realisers cannot make use of the features of specific codes.

\begin{definition}
    We obtain \emph{uniform realisability} by restricting the class of potential realisers to uniform OTMs.
\end{definition}

As we will see in Theorem \ref{V=L uniform is the same as realisable}, if $\Ve=\Le$, then the notions of uniform realisability and realisability coincide. On the other hand, if $\Ve\neq\Le$ then the two notions may differ. Indeed, by Theorem \ref{Theo:notAC} there is a model of set theory in which the previous example provides a formula which is realised but not uniformly realised.

\subsection{Basic Properties of OTM-Realisability}

We are now ready to establish some basic properties of OTM-realisability. First, we simplify the conditions for set-membership and equality in our definition of realisability.

\begin{lemma}
    \label{Lemma: equality and set-membership are trivial}
    Let $X$ and $Y$ be sets.
    \begin{enumerate}
        \item \label{Lemma:RealEq1} There is a uniform realiser of $X = Y$ if and only if $X = Y$.
        \item \label{Lemma:RealEq2} There is a uniform realiser of $X \in Y$ if and only if $X \in Y$.
    \end{enumerate}
\end{lemma}
\begin{proof}    
For \ref{Lemma:RealEq1}, the left-to-right direction follows from the definition of the realisability relation $\Vdash$ and Lemma \ref{Lemma:IsoCodEq}. For the right-to-left direction assume that $X = Y$. Let $a$ and $b$ two codes for $X$ based on the pre-codes $A$ and $B$, respectively. Let $c$ be a code for a program that returns the code of the algorithm $\mathtt{Decode}(\cdot,\cdot,\cdot)$ from Lemma \ref{Lemma:BasicAlgo} and the ordinal $\mathrm{Decode}(a,b,\rho_a)$. We will show that $\langle c, \emptyset \rangle \Vdash X = Y$. To this end, we define $f(\alpha)=\mathrm{Decode}(a,b,\alpha)$ for every $\alpha \in \essdom(a)$. It suffices to show that $f$ is a bijection from $\essdom(a)$ to $\essdom(b)$ to see that it is a code-isomorphism.

First, to see that $f$ is an injection, let $\alpha\in \essdom(a)$. Then $d_a(\alpha)\in X=Y$, and there is $\beta \in \dom(d_b)$ such that $d_b(\beta)=d_a(\alpha)\in X=Y$. But then, by Lemma \ref{Lemma:BasicAlgo}, $f(\alpha)=\beta$ and, since $d_b(\beta)\in \mathrm{tc}(\{d_b(\rho_b)\})=\mathrm{tc}(\{Y\})$, we have $\beta \in \essdom(b)$ by definition. Now let $\alpha\neq \beta$ be in $\essdom(a)$, then by injectivity of $d_a$ we have that $d_a(\alpha)\neq d_a(\beta)$. Finally by Lemma \ref{Lemma:BasicAlgo} $d_b(f(\alpha))=d_a(\alpha)\neq d_a(\beta)=d_b(f(\beta))$. So, $f$ is injective.

Second, to see that $f$ is surjective, let $\beta \in \essdom(b)$. Then, since $X=Y$, there is $\alpha\in \essdom(a)$ such that $d_a(\alpha)=d_b(\beta)$. But then, by Lemma \ref{Lemma:BasicAlgo}, $f(\alpha)$ is such that $d_b(f(\alpha))=d_a(\alpha)=d_b(\beta)$, and since $d_b$ is a bijection, we have that $f(\alpha)=\beta$ as desired. This concludes the proof of \ref{Lemma:RealEq1}.

For \ref{Lemma:RealEq2}, first assume that there is a realiser $\langle c, P \rangle \Vdash X \in Y$. Let $a$ and $b$ be any two codes of $X$ and $Y$, respectively. By definition of the realisability relation $\Vdash$, the program $c$ computes an ordinal $\alpha$ such that $d_b(\alpha) = X$ and a realiser of $X = d_b(\beta)$ so the claim follows from the previous case. For the right-to-left direction it is enough to note that, using $\mathtt{Decode}(\cdot,\cdot,\cdot)$, one can easily compute the ordinal $\alpha$ such that $d_b(\alpha)\in d_b(\rho_b)$ and the realiser of $X=d_b(\alpha)$. This concludes the proof of \ref{Lemma:RealEq2}.
\end{proof}

\begin{lemma}{\label{computing minimal codes}}
    There is an OTM-program $P_{\text{min}}$ such that, for every constructible code $a$ of a set $X \in \Le$, $P_{\text{min}}(c)$ computes the $<_{\Le}$-minimal code of $X$, where $<_{\Le}$ is the canonical well-ordering of $\Le$.
\end{lemma}
\begin{proof}
    $P_{\text{min}}$ works by successively writing codes for all constructible levels on the tape (see Koepke \cite{Koepke} for the details on how an OTM can write codes for levels of the constructible hierarchy). For each such level, 
    it checks whether it contains a code for the set coded by $a$ (i.e., $X$). As soon as such an $\Le$-level $\Le_{\alpha}$ has been found, it searches through $\Le_{\alpha}$ in the $<_{\Le}$-ordering to determine the $<_{\Le}$-minimal such code and write it to the output tape.  
\end{proof}

\begin{theorem}\label{V=L uniform is the same as realisable}
    If $\Ve=\Le$ then the following statements are equivalent:
    \begin{enumerate}
        \item $\varphi$ is realised;
        \item $\varphi$ is uniformly realised.
    \end{enumerate} 
\end{theorem}
\begin{proof}
Clearly, if a statement $\phi$ is uniformly realised, then it is realised. 

 For the reverse direction, suppose that $\phi$ is realised, and that $\langle c,P\rangle$ realises $\phi$. 
We may assume inductively that the statement holds for all subformulas of $\phi$; we also note that the only point in which the definitions of realiser and uniform realiser differ is in the case of  quantification, so that we can assume without loss of generality that $\phi$ is (i) $\forall{x}\psi$ or (ii) $\exists{x}\psi$ for some formula $\psi$. 

In these cases, in order to obtain a uniform realiser from the plain realiser $\langle c,P\rangle$, we simply start by applying $P_{\text{min}}$ to the given set before passing it over to $\langle c,P \rangle$. In this way, all codes for sets will be replaced by the corresponding $<_{\Le}$-minimal codes before further processing, so that the results of the computations will become independent of the choice of codes.
\end{proof}

\begin{lemma}\label{Lemma:Sigma1Fin}
Let $\phi(\mathbf{x})$ be an infinitary $\Sigma^{\omega}_1$-formula, and $\bar{X}$ be a sequence of the same length as $\mathbf{x}$.
The formula $\phi[\bar{X}/\mathbf{x}]$ is uniformly realised if and only if $\phi[\bar{X}/\mathbf{x}]$ is true.
\end{lemma}
\begin{proof}
The proof is an induction on the complexity of $\varphi(\mathbf{x})$. The cases where $\varphi$ is ``$X \in Y$'' or ``$X = Y$'' follow from \Cref{Lemma: equality and set-membership are trivial}. If $\varphi = \bot$, then the claim is trivial because $\bot$ is both never realised and false. The case for implication follows directly from the induction hypothesis and the definitions. 

If $\varphi(\mathbf{x})=\bigvee_{\beta < \gamma} \varphi_\beta$, assume that $\langle c, P \rangle$ uniformly realises $\bigvee_{\beta < \gamma} \varphi_\beta[\bar{X}/\mathbf{x}]$. Then $c$ is the code of a program that computes an ordinal $\beta < \gamma$ and a realiser of $\varphi_\beta[\bar{X}/\mathbf{x}]$. By the induction hypothesis, $\varphi_\beta[\bar{X}/\mathbf{x}]$ is true and, therefore, the disjunction  $\varphi$ is true as well. If $\bigvee_{\beta \in \gamma} \varphi_\beta[\bar{X}/\mathbf{x}]$ is true, then $\varphi_\beta[\bar{X}/\mathbf{x}]$ is true for some $\beta < \gamma$. By induction hypothesis, $\varphi_\beta[\bar{X}/\mathbf{x}]$ is realised by some potential realiser $\langle c',P'\rangle$. Let $P$ be a code for the pair $\langle \beta ,\langle c',P'\rangle\rangle$, and $c$ be a code for a program that given $0$ as input just returns $P$. Then $\langle c,P\rangle$ uniformly realises $\varphi[\bar{X}/\mathbf{x}]$.

If $\varphi(\mathbf{x})=\bigwedge_{\beta < \gamma} \varphi_\beta$, assume that $\langle c, P \rangle$ uniformly realises $\bigwedge_{\beta < \gamma} \varphi_\beta$. Then $c$ codes a program that, given an ordinal $\beta < \gamma$, computes a realiser of $\varphi_\beta[\bar{X}/\mathbf{x}]$. Hence, each formula $\varphi_\beta[\bar{X}/\mathbf{x}]$ is realised and, by induction hypothesis, true. In conclusion, the conjunction is also true. We prove the converse by contradiction. Suppose that $\bigwedge_{\beta < \gamma} \varphi_\beta[\bar{X}/\mathbf{x}]$ is true but not realised. This means that there is a $\beta < \gamma$ such that $\varphi_\beta[\bar{X}/\mathbf{x}]$ is not realised. By induction hypothesis, $\varphi_\beta[\bar{X}/\mathbf{x}]$ is false. A contradiction. 

 
Assume that $\varphi(\mathbf{x})=\exists^1 y \psi$. If $\langle c,P\rangle$ uniformly realises $\varphi[\bar{X}/\mathbf{x}]$, then $c$ returns the code of a program with parameter that returns, on input $0$, a code for a set $Y$ and a realiser of $\psi[Y/y][\bar{X}/\mathbf{x}]$. By the induction hypothesis, $\psi[Y/y][\bar{X}/\mathbf{x}]$ holds and, therefore, $\varphi$ is true. On the other hand, if $\varphi[\bar{X}/\mathbf{x}]$ is true, then there is a set $Y$ such that $\psi[Y/y][\bar{X}/\mathbf{x}]$ is true and, by induction hypothesis, there is a realiser $\langle c',P'\rangle$ of $\psi[Y/y][\bar{X}/\mathbf{x}]$. Let $P$ encode the realiser $\langle a, \langle c',P'\rangle \rangle$, where $a$ is a code for $Y$, and let $c$ be the code of a program that copies the parameter to the output tape. It is not hard to see that $\langle c,P\rangle$ uniformly realises $\varphi[\bar{X}/\mathbf{x}]$ as desired. 

Finally, let $\varphi(\mathbf{x})=\forall^{1} y\in x_{j} \psi$ for some $j<\ell(\mathbf{x})$. Assume that $\langle c, P\rangle$ uniformly realises $\varphi[\bar{X}/\mathbf{x}]$. Then $c$ is a program that, given $P$ and a code for a set $Y$ as input, returns a code for a realiser of $Y \in X_{j} \rightarrow \psi[Y/y][\bar{X}/\mathbf{x}]$. So, by induction hypothesis for every $Y\in X_{j}$ we have $\psi[Y/\mathbf{y}][\bar{X}/\mathbf{x}]$ as desired. For the right-to-left direction, assume that $\varphi[\bar{X}/\mathbf{x}]$ is true. By induction hypothesis, for every $Y\in X_{j}$ there is a realiser $\langle c_{Y},P_{\bar{Y}}\rangle$ of $\psi[Y/\mathbf{y}][\bar{X}/\mathbf{x}]$. Let $f:X_{j}\rightarrow \mathbb{N}\times \Ord$ be the function that maps $Y$ to the realiser $\langle c_{Y},P_{Y}\rangle$. Let $P$ be a code for the function $f$ (as a list of codes for the pairs $\langle a,\langle  c_{Y},P_{Y}\rangle \rangle$ where $a$ is a code of $Y$), and let $c$ be the code of a program that, given the code of $Y\in X_{j}$, searches in $P$ the pair $\langle a_{Y},\langle  c_{Y},P_{Y}\rangle \rangle$ and copies $\langle  c_{Y},P_{Y}\rangle$ to the output tape. The pair $\langle  c,P\rangle$ is a uniform realiser of $\varphi$. Indeed, given a code for $Y\in X_{j}$ the program coded by $c$ with parameter $P$ returns a code for a realiser of $\psi[Y/y][\bar{X}/\mathbf{x}]$, moreover the output of the machine is independent from the specific coding of $\bar{X}$, so the realiser is uniform. 
\end{proof}

Lemma \ref{Lemma:Sigma1Fin} can be extended to $\Sigma^{\infty}_1$-formulas. To prove this, we need the following lemma. Recall that ``$\mathbf{x} \in y$'', where $\mathbf{x}$ is a context-variable and $y$ a set-variable, is an abbreviation as introduced in \Cref{Section: Infinitary Intuitionistic Logic}. The following lemma justifies this abbreviation semantically.

\begin{lemma}\label{Lemma:sequenceBound}
    Let $\mathbf{x}$ be a context of length $\mu$ and $y$ a variable. For every set $Y$ and for all $\bar{X}$ we have that $(\mathbf{x}\in y)[\bar{X}/\mathbf{x}][Y/y]$ is realised if and only if $\bar{X} \in Y$.
\end{lemma}
\begin{proof}
    For the right-to-left direction assume that  $\bar{X}\in Y$. Let $\mu = \ell(\mathbf{X}) = \ell(\bar{X})$. Since $\bar{X}$ is a function with domain $\ell(\bar{X})$, Lemma \ref{Lemma:Sigma1Fin} implies that the following sentence is realised by some $s$:
    \begin{multline*}
        (\text{\vir{$f$ is a function with domain the ordinal $d$}}  \land  ( \bigwedge_{j'<j<\mu}(z_{j'}\in z_{j} \land z_{j}\in d \land f(z_{j})=x_{j} ) \\
        \land  \forall x ( x\in d\rightarrow \bigvee_{j<\mu}z_{j}=x) ) \land f\in y ) [\bar{X}/f][\ell(\bar{X})/d][\bar{X}/\mathbf{z}]
    \end{multline*}
    Given a code of a sequence $\bar{X}$ an OTM can easily compute a code for $\ell(\bar{X})$ because $\ell(\bar{X})=\dom(\bar{X})$. Let $r$ be a program that returns $\langle \bar{X},r' \rangle$ on input $0$, where $r'$ is a program that returns $\langle \ell(\bar{X}), r'' \rangle$ on input $0$ with $r''$ a program such that $r''(0)=\langle \bar{X},s \rangle$. Then $r$ realises $(\mathbf{x} \in  y)[\bar{X}/\mathbf{x}][Y/y]$.
    
    For the left-to-right direction, assume that $(\mathbf{x}\in y)[\bar{X}/\mathbf{x}][Y/y]$ is realised by some program $r$. Then $r$ is such that $r(0)= \langle F,r' \rangle$, where $r'$ is a program such that $r'(0)=(D, r'')$, and $r''$ is a program such that $r''(0)= \langle \bar{Z},s \rangle$ and $s$ realises
    \begin{multline*}
        [\text{\vir{$f$ is a function with domain the ordinal $d$}}  \land 
    (\bigwedge_{j'<j<\mu}(z_{j'}\in z_{j} \land z_{j}\in d \land f(z_{j})=x_{j}) \\ \land \forall x (x\in d\rightarrow \bigvee_{j<\mu}z_{j}=x)) \land f\in y ][F/f][D/d][\bar{Z}/\mathbf{z}].
    \end{multline*}
    By Lemma \ref{Lemma:Sigma1Fin}, it follows that $F=\bar{X}$ and $F \in Y$. Therefore, $\bar{X} \in Y$ as desired. 
\end{proof}

\begin{lemma}\label{Lemma:Delta0}
    Let $\phi(\mathbf{x})$ be an infinitary $\Sigma^{\infty}_1$-formula and $\bar{X}$ a sequence of length $\ell(\mathbf{x})$. The formula $\phi[\bar{X}/\mathbf{x}]$ is uniformly realised if and only if $\phi[\bar{X}/\mathbf{x}]$ is true.
\end{lemma}
\begin{proof}
    With Lemma \ref{Lemma:sequenceBound} it is easy to adapt the proof of Lemma \ref{Lemma:Sigma1Fin} for the present case.
\end{proof}

\begin{theorem}\label{Theo:SatRelationDelta0}
    Let $\varphi(\mathbf{x})$ be an infinitary $\Delta^{\infty}_0$-formula in the language of set theory and $\bar{X}$ be a sequence of length $\ell(\mathbf{x})$. There is an OTM that, given codes for $\varphi$ and $\bar{X}$, returns $1$ if $\varphi[\bar{X}/\mathbf{x}]$ is true and $0$ otherwise.
\end{theorem}
\begin{proof}
    This result is a variation on a Lemma of Koepke \cite[Lemma 4]{Koepke} for infinitary formulas. The proof is an induction on the complexity of the formula $\varphi(\mathbf{x})$. We already provided the algorithms for atomic formulas in Lemma \ref{Lemma:BasicAlgo}.

    If $\varphi(\mathbf{x})=\lnot \psi(\mathbf{x})$, then the program computes the truth value of $\psi(\bar{X})$ and flip the result. 

    If $\varphi(\mathbf{x})=\bigvee_{\beta<\alpha} \psi_\beta$, then the program recursively computes the truth values $\psi_\beta$ for every $\beta < \alpha$. The program outputs $1$ if one of the recursive instances stops with $1$; it outputs $0$ if after the search is done none of the instances halted with $1$. A similar procedure 
    works for $\bigwedge$. 

    If $\varphi(\mathbf{x})=\exists \mathbf{y}\in x_{j} \psi$ for some $j<\ell(\mathbf{x})$ and a $\Delta^{\infty}_0$-formula $\psi$, then the program goes through the sequences $\bar{Y}$ in $X_{j}$ and for each of them recursively computes the truth value of $\psi[\bar{Y}/\mathbf{y}][\bar{X}/\mathbf{x}]$. If the program finds a $\bar{Y}$ on which the recursive call returns $1$, then it returns $1$, otherwise it returns $0$. A similar proof works for $\forall$. 
\end{proof}

\begin{lemma}[Universal Realisability Program]
    \label{Lemma:UniversalRealisability}
    There is an OTM-program $\Phi$ that takes codes of a sequence $\bar{X}$ and an infinitary $\Delta^{\infty}_0$-formula $\phi$ as input and returns the code of a realiser of $\varphi[\bar{X}/\mathbf{x}]$ whenever $\varphi[\bar{X}/\mathbf{x}]$ is realised. The same holds if we substitute realisability with uniform realisability. Moreover, if $\Ve = \Le$ then the statement is true for infinitary $\Sigma^{\infty}_1$-formulas.
\end{lemma}
\begin{proof}
    We give an informal description of the program $\Phi$ that computes the realisers. The program first checks the main operator of the input-formula and then proceeds as follows:

    If $\varphi$ is ``$X\in Y$'' or ``$X=Y$'', then the program just returns the codes of the algorithms described in Lemma \ref{Lemma:BasicAlgo} to compute a realiser of $\varphi$. If $\varphi=\bot$ then the program can just return anything since $\bot$ is never realised. 

    If $\varphi=\bigvee_{\beta \in \gamma} \varphi_\beta$, then proceed as follows. By Theorem \ref{Theo:SatRelationDelta0}, the $\Delta^{\infty}_0$-satisfaction relation is computable by an OTM. For every sentence $\varphi_\beta[\bar{X}/\mathbf{x}]$ the program $\Phi$ checks whether the sentence is true or not. When it finds a $\beta$ such that $\varphi_\beta[\bar{X}/\mathbf{x}]$ holds, $\Phi$ stops the search and returns $\langle c,P \rangle$ where $P$ is a code of $\bar{X}$, and $c$ is a code of a program that returns $\beta$ and the output of a recursive call of $\Phi$ on $\varphi_\beta$ and $\bar{X}$.  

    If $\varphi=\bigwedge_{\beta \in \gamma} \varphi_\beta$, then the program returns the pair $\langle c,P\rangle$ where $c$ is a code for $\Phi$ and $P$ is a code for $\bar{X}$. 

    If $\varphi(\mathbf{x}) =\exists \mathbf{y} \in x_{j} \psi$ for some $j < \ell(\mathbf{x})$ and some $\Delta^{\infty}_0$-formula $\psi$, then the program returns $\langle c,P \rangle$, where $P$ is a code of $\bar{X}$ and $c$ is a code of a program that does the following: for every sequence $\bar{Y}\in X$ of length $\mathbf{y}$, it checks whether $\psi[\bar{Y}/\mathbf{y}][\bar{X}/\mathbf{x}]$ holds. As soon as such a $\bar{Y}$ is found, the program recursively calls $\Phi$ on $\psi[\bar{Y}/\mathbf{y}][\bar{X}/\mathbf{x}]$ and outputs a code for $\bar{Y}$ and the result of the recursive call of $\Phi$. 

    If $\varphi(\mathbf{x})=\forall \mathbf{y}\in x_{j} \psi$, then the program returns $\langle c,P \rangle$ where $P$ is a code of $\bar{X}$ and $c$ is a code for a program that takes a sequence $\bar{Y}$ as input and $P$ as parameter and runs $\Phi$ on $\psi[\bar{Y}/\mathbf{y}][\bar{X}/\mathbf{x}]$.

    The correctness of the algorithm follows by an easy induction on $\varphi$ using Lemma \ref{Lemma:Delta0}. 

    Finally, if $\Ve = \Le$, then we can treat unbounded existential quantification as follows: suppose that $\varphi(\mathbf{x})=\exists \mathbf{y} \psi$ for some $\Delta^{\infty}_0$-formula $\psi$. By a result of Koepke (implicit in \cite{Koepke} and explained in detail in \cite[Theorem 3.2.13 \& Lemma 2.5.45]{Carl2019}), $\Le$ is computably enumerable by an OTM. So the program starts an unbounded search in $\Le$ for a witness $\bar{Y}$ of $\psi[\bar{Y}/\mathbf{y}][\bar{X}/\mathbf{x}]$. If it finds one, it recursively calls itself on $\psi[\bar{Y}/\mathbf{y}][\bar{X}/\mathbf{x}]$ and outputs a code for $\bar{Y}$ and the result of the recursive call. 
\end{proof}

\section{Soundness of OTM-Realisability}
\label{Section: Soundness}

\subsection{Logic}

In this section, we will show that our notion of realisability is sound with respect to the infinitary sequent calculus that was introduced by Espíndola \cite{Espindola2018}. In fact, Espíndola defines a calculus for the language $\Lo_{\kappa^+,\kappa}$ for every cardinal $\kappa$. As we admit formulas of all ordinal lengths, we will show that OTM-realisability is sound with respect to these systems for every $\kappa$. A \textit{sequent} $\Gamma \vdash_\mathbf{x} \Delta$ is an ordered pair of sets of formulas in the infinitary language of set theory $\Lo^\in_{\infty, \infty}$, where $\mathbf{x}$ is a common context for all formulas in $\Gamma \cup \Delta$. The rules checked in Propositions 28--37 are given exactly as in Espíndola \cite{Espindola2018}, Def. 1.1.1.

\begin{definition}
    Let $\Gamma \cup \Delta$ be a set of formulas in the infinitary language of set theory $\Lo^\in_{\infty, \infty}$. A sequent $\Gamma \vdash_\mathbf{x} \Delta$ is \emph{realised} if the universal closure of $\bigwedge \Gamma \rightarrow \bigvee \Delta$ is realised. If $A$ and $B$ are sets of sequents, then the rule $\frac{A}{B}$ is \emph{realised} if whenever all sequents in $A$ are realised, then there is some sequent in $B$ that is realised.
\end{definition}

Note that a formula $\phi$ is realised if and only if the sequent $\emptyset \vdash \phi$ is realised. To simplify notation, we will write $\efrac{A}{B}$ to denote the conjunction of both rules $\frac{A}{B}$ and $\frac{B}{A}$. 

In many of the following soundness proofs, we will need to show that certain sequents $\Gamma \vdash_\mathbf{x} \Delta$ are realised, i.e., we have to find a realiser of the corresponding formula $\forall \mathbf{x} (\bigwedge \Gamma \rightarrow \bigvee \Delta)$. In many cases these realisers will be independent of $\mathbf{x}$. In such cases, we will, for the sake of simplicity, directly describe a realiser $r$ of $\bigwedge \Gamma \rightarrow \bigvee \Delta$, when we really mean a realiser that, given any $\bar{X}$ of length $\ell(\mathbf{x})$, outputs the realiser $r$.

\begin{proposition}
    \label{Proposition: Structural Rules are realised}
    The following structural rules are realised:
    \begin{enumerate}
        \item Identity axiom (Espíndola, \cite{Espindola2018}, Def. 1.1.1, 1(a)):
            \begin{mathpar}
                \phi \vdash_{\mathbf{x}} \phi 
            \end{mathpar}
        \item Substitution rule (Espíndola, \cite{Espindola2018}, Def. 1.1.1, 1(b)):
            \begin{mathpar}
                \inferrule{\phi \vdash_{\mathbf{x}} \psi}{\phi[\mathbf{s}/\mathbf{x}] \vdash_{\mathbf{y}} \psi[\mathbf{s}/\mathbf{x}]} 
            \end{mathpar}
            where $\mathbf{y}$ is a string of variables including all variables occurring in the string of terms $\mathbf{s}$.
        \item Cut rule (Espíndola, \cite{Espindola2018}, Def. 1.1.1, 1(c)):
            \begin{mathpar}
                \inferrule{\phi \vdash_{\mathbf{x}} \psi \\ \psi \vdash_{\mathbf{x}} \theta}{\phi \vdash_{\mathbf{x}} \theta} 
            \end{mathpar}
    \end{enumerate}
\end{proposition}
\begin{proof}
    The identity axiom is trivially realised by an OTM that implements the identity map.
    
    For the substitution rule, recall that by our definition, a realiser of $\phi \vdash_{\mathbf{x}} \psi$ is in fact a realiser $r \Vdash \forall \mathbf{x} (\phi \rightarrow \psi)$. We need to find a realiser $t \Vdash \forall \mathbf{y} (\phi[\mathbf{s}/\mathbf{x}] \rightarrow \psi[\mathbf{s}/\mathbf{x}])$, i.e., $t$ takes as input some code for a sequence $\mathbf{y}$ of variables. To achieve this, find codes for the realiser $r$ and the substitution $\mathbf{s}/\mathbf{x}$. Then let $t$ be the OTM with parameters $r$ and $\mathbf{s}/\mathbf{x}$ that performs the following two steps: first, reorder the input $\mathbf{y}$ according to the substitution $\mathbf{s}/\mathbf{x}$, and then apply the parameter $r$ to compute a realiser of $\phi[\mathbf{s}/\mathbf{x}] \rightarrow \psi[\mathbf{s}/\mathbf{x}]$.
        
    For the cut rule, let $r \Vdash \forall \mathbf{x} (\phi \rightarrow \psi)$ and $s \Vdash \forall \mathbf{x} (\phi \rightarrow \theta)$. For any given input $\bar{X}$, we have that $r(\bar{X}) \Vdash \phi(\bar{X}) \rightarrow \psi(\bar{X})$ and $s(\bar{X}) \Vdash \psi(\bar{X}) \rightarrow \theta(\bar{X})$, i.e., $r(\bar{X})$ maps realisers of $\phi(\bar{X})$ to realisers of $\psi(\bar{X})$, and $s(\bar{X})$ maps realisers of $\psi(\bar{X})$ to realisers of $\theta(\bar{X})$. Let $t$ be the OTM that, given input $\bar{X}$, returns an OTM that given input $u$ returns $s(\bar{X})(r(\bar{X})(u))$. Then $t \Vdash \forall \mathbf{x} (\phi(\mathbf{x}) \rightarrow \psi(\mathbf{x}))$.
\end{proof}

\begin{proposition}
    \label{Proposition: Equality Rules}
    The following rules for equality are realised:
    \begin{enumerate}
    \item $\top \vdash_{x} x=x$ (Espíndola, \cite{Espindola2018}, Def. 1.1.1, 2(a))
    \item $(\mathbf{x}=\mathbf{y}) \wedge \phi[\mathbf{x}/\mathbf{z}] \vdash_{\mathbf{z}} \phi[\mathbf{y}/\mathbf{z}]$
        where $\mathbf{x}$, $\mathbf{y}$ are contexts of the same length and type and $\mathbf{z}$ is any context containing $\mathbf{x}$, $\mathbf{y}$ and the free variables of $\phi$. (Espíndola, \cite{Espindola2018}, Def. 1.1.1, 2(b))
    \end{enumerate}
\end{proposition}
\begin{proof}
    Finding a realiser of the first statement means finding a realiser of $\forall x (\top \rightarrow x = x)$, which is equivalent to finding a realiser of $\forall x (x = x)$. This follows directly from the fact that the algorithm for equality presented in the proof of \Cref{Lemma: equality and set-membership are trivial} is the same for any sets $X$ and $Y$.
    
    The second statement follows in a similar way as the substitution rule of \Cref{Proposition: Structural Rules are realised} using that, by the definition of realisability, realisers work on all codes of any given set.
\end{proof}

\begin{proposition}
    \label{Proposition: Conjunction Rules}
    Let $\kappa$ be a cardinal. The following conjunction rules (Espíndola, \cite{Espindola2018}, Def. 1.1.1, 3) are realised: 
    \begin{enumerate}
        \item $$\bigwedge_{i<\kappa} \phi_i \vdash_{\mathbf{x}} \phi_j,$$
        \item
            \begin{mathpar}
                \inferrule{\{\phi \vdash_{\mathbf{x}} \psi_i\}_{i<\kappa}}{\phi \vdash_{\mathbf{x}} \bigwedge_{i<\kappa} \psi_i}.
            \end{mathpar}
    \end{enumerate}
\end{proposition}
\begin{proof}
    The definition of realisability straightforwardly implies that both rules are realised: for the first rule observe that we can just extract a realiser of $\phi_j$ from a realiser of $\bigwedge_{i < \kappa} \phi_i$. For the second rule, combine the realisers $r_i \Vdash \phi \rightarrow \psi_i$ for $i < \kappa$ into a parameter $P$ and obtain a realiser of $\phi \vdash_\mathbf{x} \bigwedge_{i < \kappa} \phi_i$ by implementing an OTM program that returns the realiser of $\phi_i$ on input $i$.
\end{proof}

\begin{proposition}
    \label{Proposition: Disjunction Rule}
    Let $\kappa$ be a cardinal. The following disjunction rules (Espíndola, \cite{Espindola2018}, Def. 1.1.1, 4) are realised:
    \begin{enumerate}
        \item $$\phi_j \vdash_{\mathbf{x}} \bigvee_{i<\kappa} \phi_i$$
        \item
            \begin{mathpar}
                \inferrule{\{\phi_i \vdash_{\mathbf{x}} \theta\}_{i<\kappa}}{\bigvee_{i<\kappa} \phi_i \vdash_{\mathbf{x}} \theta}
            \end{mathpar}
        \end{enumerate}
\end{proposition}
\begin{proof}
    For the first statement, we need to realise the implication $\phi_j \rightarrow \bigvee_{i < \kappa} \phi_i$. This can be done by an OTM that, given a realiser $r_j$ of $\phi_j$, returns an OTM that returns a tuple $\langle j, r_j \rangle$ on input $0$.
    
    For the second statement, code the realisers $r_j$ for $\phi_i \vdash_\mathbf{x} \theta$, $i < \kappa$, into a parameter $P$. Then, $\bigvee_{i < \kappa} \phi_i \vdash_\mathbf{x} \theta$ is realised by the following algorithm implemented by an OTM: given a realiser $s \Vdash \bigvee_{i < \kappa} \phi_i$, compute $s(0) = \langle i, t \rangle$, such that $t \Vdash \phi_i$ and then return $r_i(t)$ by using the parameter.
\end{proof}

\begin{proposition}
    \label{Proposition: Implication Rule}
    The following implication rule (Espíndola, \cite{Espindola2018}, Def. 1.1.1, 5) is realised:
    \begin{mathpar}
        \efrac{\phi \wedge \psi \vdash_{\mathbf{x}} \eta}{\phi \vdash_{\mathbf{x}} \psi \to \eta}
    \end{mathpar}
\end{proposition}
\begin{proof}
    We have to show two directions. For the first direction, \emph{top-to-bottom}, let $r \Vdash \phi \wedge \psi \rightarrow \eta$. Now, we construct a realiser of $\phi \rightarrow (\psi \rightarrow \eta)$ as follows: given a realiser $r_\phi$ for $\phi$, output the OTM that, given a realiser $r_\psi$ for $\psi$, combines it with $r_\phi$ to obtain a realiser $r_{\phi \wedge \psi} \Vdash \phi \wedge \psi$. Then apply $r(r_{\phi \wedge \psi}) \Vdash \eta$.
    
    For the other direction, \emph{bottom-to-top}, suppose we have a realiser $r \Vdash \phi \rightarrow (\psi \rightarrow \eta)$. We obtain a realiser of $\phi \wedge \psi \rightarrow \eta$ as follows: given a realiser of $r_{\phi \wedge \psi} \Vdash \phi \wedge \psi$, compute realisers $r_{\phi} \Vdash \phi$ and $r_{\psi} \Vdash \psi$. By definition and our assumptions, $(r(r_{\phi}))(r_{\psi}) \Vdash \eta$. 
\end{proof}

\begin{proposition}
    \label{Proposition: Existential Rule}
    The following existential rule (Espíndola, \cite{Espindola2018}, Def. 1.1.1, 6) is realised:
    \begin{mathpar}
        \efrac{\phi \vdash_{\mathbf{x} \mathbf{y}} \psi}{\exists \mathbf{y}\phi \vdash_{\mathbf{x}} \psi}
    \end{mathpar}
    where no variable in $\mathbf{y}$ is free in $\psi$.
\end{proposition}
\begin{proof}
    For the first direction, assume that $r \Vdash \forall \mathbf{xy} (\phi \rightarrow \psi)$. We have to find a realiser $t \Vdash \forall \mathbf{x} (\exists \mathbf{y} \phi \rightarrow \psi)$. Let $t$ be an implementation of the following algorithm: given some sequence $\bar{X}$ and a realiser $r_\exists \Vdash \exists \mathbf{y} \psi$. Compute from $r_\exists$ a code for some $\bar{Y}$ and a realiser $r_\phi \Vdash \phi$. Then calculate $(r(\bar{X}\bar{Y}))(r_\phi)$. This is a realiser of $\psi$.
    
    For the other direction, assume that $r \Vdash \forall \mathbf{x} (\exists \mathbf{y} \phi \rightarrow \psi)$. We construct a realiser $t \Vdash \forall \mathbf{xy} (\phi \rightarrow \psi)$. So let sequences $\bar{X}$, $\bar{Y}$ be given, and assume that we have a realiser $s \Vdash \phi(\bar{X}\bar{Y})$. We can compute a realiser $r_\exists$ for $\exists \mathbf{y} \phi(\bar{X})$ as the OTM that returns $\bar{X}$ and $s$. We can then return $(r(\bar{X}))(r_\exists)$, which is a realiser of $\psi$.
\end{proof}

\begin{proposition}
    \label{Proposition: Universal Rule}
    The following universal rule (Espíndola, \cite{Espindola2018}, Def. 1.1.1, 7) is realised:
    \begin{mathpar}
        \efrac{\phi \vdash_{\mathbf{x} \mathbf{y}} \psi}{\phi \vdash_{\mathbf{x}} \forall \mathbf{y} \psi}
    \end{mathpar}
    where no variable in $\mathbf{y}$ is free in $\phi$.
\end{proposition}
\begin{proof}
    For the \emph{top-to-bottom}-direction, assume that $r \Vdash \forall \mathbf{x} \mathbf{y} (\phi \rightarrow \psi)$. We have to find a realiser $t \Vdash \forall \mathbf{x} (\phi \rightarrow (\forall \mathbf{y} \psi))$. If $\bar{X}$ is a sequence and $r_\phi$ a realiser of $\phi$, then the OTM that takes some $\bar{Y}$ as input and returns $r(\bar X \bar Y)(r_\phi)$ is a realiser of $\forall \mathbf{y} \psi$. Call this realiser $r_{\bar X}$. Then, the OTM which takes some $\bar{X}$ as input and then returns $r_{\bar X}$ is a realiser of $\forall \mathbf{x} (\phi \rightarrow \forall \mathbf{y} \psi)$.
    
    For the \emph{bottom-to-top}-direction, assume that $r \Vdash \forall \mathbf{x} (\phi \rightarrow (\forall \mathbf{y} \psi))$. Then $\forall \mathbf{x} \mathbf{y} (\phi \rightarrow \psi)$ can be realised by the OTM that operates as follows: given sequences $\bar{X}$ and $\bar{Y}$ as input, return the OTM that, given a realiser $s \Vdash \phi$, returns $((r(\bar X))(s))(\bar Y) \Vdash \psi$.
\end{proof}

\begin{proposition}
    \label{Proposition: Small Distributivity Axiom}
    The small distributivity axiom (Espíndola, \cite{Espindola2018}, Def. 1.1.1, 8) is realised:
    $$\bigwedge_{i<\kappa} (\phi \vee \psi_i) \vdash_{\mathbf{x}} \phi \vee \left(\bigwedge_{i<\kappa} \psi_i\right)$$
    for each cardinal $\kappa$.
\end{proposition}
\begin{proof}
    It is enough to construct an OTM that transforms a realiser of $\bigwedge_{i<\kappa} (\phi \vee \psi_i)$ into a realiser of $\phi \vee \left(\bigwedge_{i<\kappa} \psi_i\right)$. So let $r \Vdash \bigwedge_{i < \kappa} (\phi \vee \psi_i)$. The OTM proceeds as follows: first, search for a realiser of $\phi$ by going through all $i < \kappa$. As soon as a realiser of $\phi$ is found, we are done. If no realiser of $\phi$ is found, then we have, in fact, realisers for every $\psi_i$ for $i < \kappa$ and can therefore construct a realiser of $\bigwedge_{i < \kappa} \psi_i$.
\end{proof}

Recall that a \emph{bar} is an upwards-closed subset of a tree that intersects every branch of the tree.

\begin{proposition}
    \label{Proposition: Dual Distributivity}
    The dual distributivity rule (Espíndola, \cite{Espindola2018}, Def. 1.1.1, 9) is realised, i.e.:
    \begin{mathpar}
    \inferrule{
            \bigwedge_{g \in \gamma^{\beta+1}, g|_{\beta}=f} \phi_{g} \vdash_{\mathbf{x}} \phi_{f} \\ \beta<\kappa, f \in \gamma^{\beta} \\\\ 
            \phi_{f} \dashv \vdash_{\mathbf{x}}         \bigvee_{\alpha<\beta}\phi_{f|_{\alpha}} \\ \beta < \kappa, \text{ limit }\beta, f \in \gamma^{\beta}
        }
        {
            \bigwedge_{f \in B} \bigvee_{\beta<\delta_f}\phi_{f|_{\beta+1}} \vdash_{\mathbf{x}} \phi_{\emptyset}
        }
    \end{mathpar}
    \\
    holds whenever $\gamma$ is a cardinal strictly below $\kappa^+$, where $B$ denotes the subset of $\gamma^{<\kappa}$ that contains the minimal elements of some bar of the tree $\gamma^{\kappa}$ and, for $f\in B$, $\delta_{f}$ denotes the level of $f$.
\end{proposition}
\begin{proof}
    Our goal is to construct a realiser of 
    $$
        \bigwedge_{f \in B} \bigvee_{\beta<\delta_f}\phi_{f|_{\beta+1}} \vdash_{\mathbf{x}} \phi_{\emptyset},
    $$
    i.e., we have to construct an OTM that computes a realiser of $\phi_\emptyset$ from a realiser of $\bigwedge_{f \in B} \bigvee_{\beta<\delta_f}\phi_{f|_{\beta+1}}$. In doing so, we can use realisers of the antecedents of the rule: we denote by $r_1$ a realiser of the first antecedent and by $r_2^{\vdash}$ a realiser of the forward direction of the second antecedent.
    
    We claim that \Cref{Algo:Walking} describes a desired realiser, and we will now prove that it terminates with a realiser of $\phi_\emptyset$ when run on a realiser of $\bigwedge_{f \in B} \bigvee_{\beta<\delta_f}\phi_{f|_{\beta+1}}$. In fact, we will prove the contrapositive. So suppose that the OTM described by \Cref{Algo:Walking} does not terminate, i.e., it either loops or crashes. 
    
    First, assume that the machine loops. As our algorithm constructs $r_{(-)}$ in a monotone way, the partial function $r_{(-)}$ must stabilise before the loop. Hence, the algorithm must loop through lines 15 and 16: otherwise we would (eventually) still alter $r_{(-)}$ (lines 11--13) or contradict the well-foundedness of the ordinal numbers (lines 5--9). However, looping through lines 15 and 16 means that we build up a sequence $f$ that will eventually reach length $\kappa$. But then the operation of selecting a direct successor in line 15 will crash, a contradiction to the machine's looping.
    
    Secondly, assume that the machine crashes. It is easy to see that this must happen in line 15 as all other operations are well-defined (using the case distinctions and assumptions on $r$, $r_1$ and $r_2^\vdash$). A crash in line 15, however, will only occur if $f$ has reached length $\kappa$. This means that we have constructed a branch through $\gamma^{<\kappa}$ that does not intersect the bar $B$, a contradiction. \qedhere
       
    \begin{algorithm}[ht]
        \LinesNumbered
            \KwIn{A realiser $r$ for $\bigwedge_{f \in B} \bigvee_{\beta<\delta_f}\phi_{f|_{\beta+1}}$}
            \KwOut{A realiser $r_\emptyset \Vdash \phi_\emptyset$.}
            
            From $r$ extract a set $C \subseteq \gamma^{<\kappa}$ and a partial function $r_{(-)}: \gamma^{<\kappa} \to V$ such that $r_f \Vdash \phi_f$ for all $f \in C$. \\
            Let $f = \emptyset$. \\
            \While{$r_\emptyset$ is undefined}
            {
                \If{$f \in \dom(r_{(-)})$} 
                {
                    \If{$f$ is of successor length $\alpha + 1$}
                    {
                        Set $f := f|_\alpha$.
                    }
                    \If{$f$ is of limit length $\beta$}
                    {
                        Calculate $r_2^\vdash(r_f)$ and extract from this $r_{f|_\alpha}$ for some $\alpha < \beta$. \\
                        Set $f := f|_\alpha$.
                    }
                }
                \Else 
                {
                    \If{$g \in \dom(r_{(-)})$ for all direct successors $g$ of $f$}
                    {
                        Combine the $r_g$ into a realiser $r' \Vdash \bigwedge_{g \in \gamma^{\beta+1}, g|_{\beta}=f} \phi_{g}$. \\
                        Set $r_f := r_1(r')$.
                    }
                    \Else
                    {
                        Select a direct successor $g \in \gamma^{<\kappa}$ of $f$ such that $r_g$ is undefined. \\
                        Set $f := g$. 
                    }
                }
            }
            Return $r_\emptyset$.
            
    		\caption{$\mathrm{Walking}(r)$, walking through the tree \label{Algo:Walking}}
    \end{algorithm}
\end{proof}

\begin{proposition}
    \label{Proposition: Transfinite Transitivity Rule}
    The transfinite transitivity rule (Espíndola, \cite{Espindola2018}, Def. 1.1.1, 10) is realised:
    \begin{mathpar}
    \inferrule{\phi_{f} \vdash_{\mathbf{y}_{f}} \bigvee_{g \in \gamma^{\beta+1}, g|_{\beta}=f} \exists \mathbf{x}_{g} \phi_{g} \\ \beta<\kappa, f \in \gamma^{\beta} \\\\ \phi_{f} \dashv \vdash_{\mathbf{y}_{f}} \bigwedge_{\alpha<\beta}\phi_{f|_{\alpha}} \\ \beta < \kappa, \text{ limit }\beta, f \in \gamma^{\beta}}{\phi_{\emptyset} \vdash_{\mathbf{y}_{\emptyset}} \bigvee_{f \in B}  \exists_{\beta<\delta_f}\mathbf{x}_{f|_{\beta +1}} \bigwedge_{\beta<\delta_f}\phi_{f|_{\beta+1}}}
    \end{mathpar}
    \\
    for each cardinal $\gamma<\kappa^+$, where $\mathbf{y}_{f}$ is the canonical context of $\phi_{f}$, provided that, for every $f \in \gamma^{\beta+1}$,  $FV(\phi_{f}) = FV(\phi_{f|_{\beta}}) \cup \mathbf{x}_{f}$ and $\mathbf{x}_{f|_{\beta +1}} \cap FV(\phi_{f|_{\beta}})= \emptyset$ for any $\beta<\gamma$, as well as $FV(\phi_{f}) = \bigcup_{\alpha<\beta} FV(\phi_{f|_{\alpha}})$ for limit $\beta$. Here $B \subseteq \gamma^{< \kappa}$ consists of the minimal elements of a given bar over the tree $\gamma^{\kappa}$, and the $\delta_f$ are the levels of the corresponding $f \in B$.
\end{proposition}
\begin{proof}
    The proof of this proposition is a simplification of the proof of \Cref{Proposition: Dual Distributivity}: again, we need to search through the tree until we hit the bar. This time, however, we are starting from a realiser of $\phi_\emptyset$. Then use the first assumption to compute realisers at successor levels and use the second assumption to compute realisers at limit levels. By the definition of the bar $B$, this procedure must at some point reach some node contained in $B$, and we have found the desired realiser. Note that this procedure does not require any backtracking as in the previous proposition and, therefore, the formalisation of the desired OTM is straightforward.
\end{proof}

Espíndola's \cite{Espindola2018} system of \emph{$\kappa$-first-order logic} is axiomatised in the sequent calculus by the rules mentioned in \Cref{Proposition: Equality Rules,Proposition: Conjunction Rules,Proposition: Disjunction Rule,Proposition: Implication Rule,Proposition: Existential Rule,Proposition: Universal Rule,Proposition: Small Distributivity Axiom,Proposition: Dual Distributivity,Proposition: Transfinite Transitivity Rule}. 

\begin{corollary}
    \label{Corollary: Soundness of kappa-first-order logic}
    OTM-realisability is sound with respect to $\kappa$-first-order logic for every cardinal $\kappa$. The same holds if we substitute realisability with uniform realisability.
\end{corollary}
\begin{proof}
    This is just the combination of \Cref{Proposition: Equality Rules,Proposition: Conjunction Rules,Proposition: Disjunction Rule,Proposition: Implication Rule,Proposition: Existential Rule,Proposition: Universal Rule,Proposition: Small Distributivity Axiom,Proposition: Dual Distributivity,Proposition: Transfinite Transitivity Rule}.
\end{proof}

Before moving on to questions of which set theory is OTM-realised in the next section, we will take a brief moment to note a few logical properties of our notion of realisability.

\begin{proposition}[Semantic Disjunction and Existence Properties]
    If $\phi \vee \psi$ is realised, then $\phi$ is realised or $\psi$ is realised. More generally, if $\bigvee_{i < \kappa} \phi_i$ is realised, then there is some $i < \kappa$ such that $\phi_i$ is realised.
    
    If $\exists^\alpha \mathbf{x} \phi(\mathbf{x})$ is realised, then there is a sequence $\bar X$ such that $\phi(\bar X)$ is realised.
\end{proposition}
\begin{proof}
    This follows from the definition of realisability for disjunctions and existential quantifiers.
\end{proof}

We say that a formula of propositional logic is \textit{(uniformly) OTM-realisable} if every substitution instance of the formula is (uniformly) OTM-realisable.

\begin{theorem}
    There is a formula of propositional logic which is uniformly OTM-realisable but not a consequence of intuitionistic propositional logic.
\end{theorem}
\begin{proof}
   Let $\varphi$ be the formula $\lnot p\lor \lnot q$. The \emph{Rose formula} (due to Rose \cite{Rose1953}) is the following sentence:
    $$((\lnot \lnot \varphi\rightarrow \varphi)\rightarrow (\lnot \varphi \lor \lnot \lnot \varphi))\rightarrow (\lnot \varphi \lor \lnot \lnot \varphi).$$
    The Rose formula is OTM-realisable but it is not a consequence of intuitionistic propositional logic. This is shown in just the same way as in the arithmetical case (see Plisko's survey \cite[Section 6.1]{Plisko2009}).
\end{proof}

A similar result holds with respect to first-order logic. 

\begin{theorem}
    Each substitution instance of the following Markov's principle, formulated in first-order logic, is uniformly OTM-realisable:
    $$
        (\forall x (P(x) \vee \neg P(x)) \wedge \neg \neg \exists x P(x)) \rightarrow \exists x P(x)
    $$
    However, this formula is not a theorem of intuitionistic first-order logic.
\end{theorem}
\begin{proof}
    The realisability of every instance of Markov's principle 
    follows by providing an OTM that executes a bounded search for a witness of $P(x)$ within some big enough parameter $\Ve_\alpha$. A standard argument using Kripke semantics shows that Markov's principle is not a theorem of intuitionistic first-order logic.\footnote{Take the Kripke model based on the frame $K := \set{v, w}$ with $v < w$ and domains $D_v := \set{a}$, $D_w := \set{a,b}$, and valuate the predicate $P$ to be true only of the element $b$ at node $w$. It is then straightforward to check that $v \Vdash \forall x (P(x) \vee \neg P(x))$, $v \Vdash \neg \neg \exists x P(x)$ but $v \not \Vdash \exists x P(x)$.}
\end{proof}

These two results show that the propositional and first-order logics of OTM-realisability are stronger than intuitionistic propositional and intuitionistic first-order logic, respectively. 

\subsection{Set Theory}
In this section, we will study the realisability of various axioms of set theory. Some of these statements were already proved for uniform realisability of finitary logic by the first author \cite{Carl2019}; we include their proofs here for the sake of completeness. We will begin by proving that our notion of uniform realisability realises the axioms of the infinitary version of Kripke-Platek set theory (for the finite version, see  also \cite[Proposition 9.4.7]{Carl2019}).

\begin{definition}
    We define \textit{infinitary Kripke-Platek set theory}, denoted by $\Lo^\in_{\infty, \infty}$-$\mathrm{KP}$, on the basis of intuitionistic $\kappa$ first-order logic for every cardinal $\kappa$ with the following axioms and axiom schemata:
    \begin{enumerate}
        \item Axiom of extensionality:  $$\forall x\forall y (\forall z ((z\in x\rightarrow z\in y) \land (z\in y\rightarrow z\in x)) \rightarrow x=y),$$
        \item Axiom of empty set: $$\exists x \forall y (y\notin x);$$
        \item Axiom of pairing: $$\forall x \forall y \exists z \forall w ((w\in z\rightarrow( w=x\lor w=y)) \land ((w=x\lor w=y)\rightarrow w\in z)),$$
        \item Axiom of union: 
        $$\forall x \exists y ((y\in x \rightarrow \exists z (y\in z \land z\in x))\land (\exists z (y\in z \land z\in x) \rightarrow y\in x)),$$
        \item Axiom schema of induction:
        $$\forall \mathbf{y}((\forall x (\forall z\in x \varphi[z/x]) \rightarrow \varphi)\rightarrow \forall x \varphi),$$
        for every infinitary formula $\varphi(x,\mathbf{y})$,
        \item Axiom schema of $\Delta^{\infty}_0$-separation:        $$\ \forall \mathbf{y}\forall v \exists z \forall u ((u\in z\rightarrow (u\in v \land \varphi[u/x]))\land ((u\in v \land \varphi[u/x])\rightarrow  u\in z)),$$
        for every infinitary $\Delta^{\infty}_0$-formula $\varphi(x,\mathbf{y})$, and,
        \item Axiom schema of $\Delta^{\infty}_0$-collection:
    $$\forall \mathbf{z}(\forall x \exists y  \varphi \rightarrow \forall w \exists w' \forall x\in w\exists y\in w' \varphi),$$
    for every infinitary $\Delta^{\infty}_0$-formula $\varphi(x,y,\mathbf{z})$.
    \end{enumerate}
\end{definition}

The axiom schema of \textit{full collection} is obtained from $\Delta^{\infty}_0$-collection by allowing arbitrary infinitary formulas.

\begin{theorem}
    The axioms of infinitary Kripke-Platek set theory are uniformly realised. Moreover, the axiom schema of full collection is uniformly realised.
\end{theorem}
\begin{proof}
    It is straightforward to construct realisers for the axioms of empty set, pairing, extensionality and union. For each of the remaining axioms, we will now give an informal description of an algorithm that realises it.
    
    First, consider the axiom schema of induction. For every infinitary formula $\varphi(x,\mathbf{y})$ we have to show that the corresponding instance is realisable: 
    $$
        \forall \mathbf{y} (\forall x (\forall z \in x \varphi[z/x] \rightarrow \varphi) \rightarrow \forall x \varphi).
    $$  
    It is sufficient to describe a program $c$ that takes a code $a$ for a $\bar Y$, a code $b$ for $X$ and a realiser $r$ for $\forall x (\forall z \in x \varphi[z/x]) \rightarrow \varphi$ as input, and computes a realiser $\phi[X/x]$. A realiser of the corresponding instance of the induction axiom is then obtained through currying: transform $c(\cdot,\cdot,\cdot)$ into $\lambda a. \lambda b. \lambda r. c(a,b,r)$.
    
    We will describe a procedure to obtain a realiser of $\phi[X/x,\bar Y/\mathbf{y}]$ by recursively building up realisers in the transitive closure of $X$. To do so, the program $c$ uses two auxiliary tapes, $\mathtt{Done}$ and $\mathtt{Realisers}$, to keep track of its progress. The tape $\mathtt{Done}$ is used to keep track of the sets $Z$ in the transitive closure of $X$, for which a realiser of $\varphi[Z/x,\bar{Y}/\mathbf{y}]$ has already been computed. Meanwhile, $\mathtt{Realisers}$ is used as a queue on which the realisers for all sets in $\mathtt{Done}$ are saved. Now, until all members of the essential domain of $b$ are in $\mathtt{Done}$, the program keeps running through the elements of the essential domain of $b$. For each $\alpha$ in the essential domain of $b$, the program checks whether all its members are contained in $\mathtt{Done}$. If so, the program uses the realisers contained in $\mathtt{Realisers}$ to compute a realiser of $\forall z \in d_b(\alpha) \varphi[z/x,\bar{Y}]$. Note that this can easily be done by using the pair $\langle c',\{b,f\}\rangle$, where $f$ is a code of the current content of $\mathtt{Realisers}$ and $c'$ is the program that searches through $f$ for a realiser of the set whose code is given as input. Then the program computes $r(d)$ to obtain a realiser $r'$ of $\varphi[d_b(\alpha)/x, \bar{Y}/\mathbf{y}]$ and saves this pair $\langle \alpha,r'\rangle$ in $\mathtt{Realisers}$. Once realisers for every element of the essential domain of $b$ are computed, the program computes a realiser of $\varphi[X/x, \bar{Y}/\mathbf{y}]$ in the same way and returns it.
    
    Second, consider the axiom schema of $\Delta^{\infty}_0$-separation. By Lemma \ref{Lemma:Delta0} and Lemma \ref{Lemma:UniversalRealisability} it is enough to show that for every $X$ and $\bar{Y}$, the set $\{Z\in X\mid \varphi[Z/x,\bar{Y}/\mathbf{y}]\}$ is computable. But this follows from Theorem \ref{Theo:SatRelationDelta0} and Lemma \ref{Lemma:BasicAlgo}.
    
    Finally, consider the axiom schema of collection. For every infinitary $\varphi(x,y,\mathbf{z})$ we have the axiom
    $$
    \forall \mathbf{z}(\forall x \exists y  \varphi \rightarrow \forall w \exists w' \forall x\in w\exists y\in w' \varphi).
    $$
    Let $\bar{Z}$ be a sequence coded by some $a$. It suffices to show that there is a program that given realiser $r$ of $\forall x \exists y  \varphi $, and a code $b$ for a set $W$, computes a code for a set $W'$ such that $\forall x\in W\exists y\in W' \varphi$. Our program uses an auxiliary queue $\mathtt{Codes}$. The program searches through $b$ and for every $\alpha \in \dom(d_b)$ such that $d_b(\alpha)\in d_b(\rho_b)$ uses $r$ to compute a code for a set $Y$ such that $\varphi[d_b(\alpha)/x,Y/y,\bar{Z}/\mathbf{z}]$ is realised and saves the computed code in $\mathtt{Codes}$. The claim follows by Lemma \ref{Lemma:BasicAlgo}.
\end{proof}

The proof of the following lemma is analogous to a proof by the first author (\cite[Lemma 8.6.3]{Carl2019}), which, in turn, followed the ``cardinality method'' of Hodges, see Hodges \cite{Hodges}, Lemma 3.2.

\begin{lemma}\label{Lemma:OTMNotIcreaseCardinality}
    Let $\Phi$ be an $\mathrm{OTM}$-program. Then for every binary sequence $s$ of size $\kappa\geq \omega$, we have the following: if $\Phi$ halts on $s$, then $|\Phi(s)|<\kappa^{+}$. Moreover, there is $\gamma<\kappa^+$ such that $\Phi(s)\in \Le_\gamma[s]$.
\end{lemma}
\begin{proof}
    Assume $\kappa\geq \omega$ and that $\Phi(s)$ halts. Note that $\Phi(s)$ cannot be longer than the halting time $\lambda$ of $\Phi$ on input $s$. Indeed, $\Phi$ will need at least one step in order to compute each bit of the output.  
    
    Let $\varphi(x,y,z)$ be the $\Delta^{\omega}_0$-formula expressing that $x$ is the computation of the program $y$ on input $z$. Consider the $\Sigma^{\infty}_1$-sentence $\exists x \varphi(x,\Phi,s)$ expressing the fact that there is a computation of $\Phi$ on $s$. Let $\delta=\max\{\aleph_0, \kappa, \lambda\}$. Note that $\Le_{\delta^{+}}[s]\models \exists x \varphi(x,\Phi,s)$. Indeed, $\Le_{\delta^{+}}[s]$ contains the computation $\sigma$ of $\Phi$ over $s$, and $\varphi$ is absolute between transitive models. 
    
    Now, let $H$ be the Skolem hull of $\kappa\cup\{s\}$ in $\Le_{\delta^{+}}[s]$. Then, $H$ has size $\kappa$ and so does the transitive collapse $M$ of $H$. But then $M$ is a transitive model of $\exists x \varphi(x,\Phi,s)$. Therefore, there is $\sigma'\in M$ such that $M\models \varphi(\sigma',\Phi,s)$, and since $\varphi$ is a $\Delta^{\infty}_0$-formula it is absolute between transitive models of set theory. Therefore, $\sigma'$ is the computation of $\Phi$ with input $s$ of size at most $\kappa$. Finally, $\lambda \geq \kappa$ and by the Condensation Lemma, $M=\Le_\gamma[s]$ for some $\gamma<\delta^{+}=\kappa^{+}$. So, $\Phi(s) \in \Le_\gamma[s]$.
\end{proof}

\begin{lemma}[{\cite[Proposition 9.4.3]{Carl2019}} ]
    \label{Lemma:Pow}
    The axiom of power set is not realised. 
\end{lemma}
\begin{proof}
    A realiser of the power set axiom is a pair $\langle c,P \rangle$ where $c$ is the code of a program with parameter $P$ that takes a code of a set $X$ as input and returns a code for $\wp(X)$. Now let $a$ be any code for the set $P$. By Lemma \ref{Lemma:OTMNotIcreaseCardinality}, the output of program $c$ with input $a$ and parameter $P$ has size ${<}|P|^{+}$. Therefore, the result of the computation cannot be a code of $\wp(P)$.
\end{proof}

\begin{corollary}
    There are sentences $\varphi$ and $\chi$ in the language of set theory such that $\varphi \rightarrow \chi$ is uniformly realised but $\varphi \rightarrow \chi$ is false.
\end{corollary}
\begin{proof}
    Let $\varphi$ be the power set axiom and $\chi$ be the sentence $\exists x \ x \neq x$. It is then trivial that $\varphi \rightarrow \chi$ is uniformly realised by the previous Lemma \ref{Lemma:Pow}. However, $\varphi \rightarrow \chi$ is false since the power set axiom holds and $\exists x \ x \neq x$ is false.
\end{proof}

Lubarsky \cite{Lubarsky2002} suggested the following infinity axiom in the context of intuitionistic Kripke-Platek set theory:
\begin{equation*}
    \exists x (\emptyset \in x \wedge (\forall y \in x \ y \cup \set{y} \in x) \wedge (\forall y\in x \ (y = \emptyset \vee \exists z \in x \ y = z \cup \set{z}))).
    \tag{Infinity}
\end{equation*}

\begin{proposition}
    \label{Proposition:Infinity realised}
    The infinity axiom is uniformly realised.
\end{proposition}
\begin{proof}
    The infinity axiom is a $\Sigma^{\infty}_1$-formula and, therefore, uniformly realised by \Cref{Lemma:Delta0}.
\end{proof}


The following theorem is a variant of a result of the first author (\cite[Proposition 3]{Carl2018A}).

\begin{theorem}[{cf. \cite[Proposition 3]{Carl2018A}}]
    Assume $\Ve = \Le$. If $\phi$ is a true $\Pi_2$-formula, then $\phi$ is uniformly realised. 
\end{theorem}
\begin{proof}
  Let $\phi$ be the $\Pi_{2}$-statement $\forall{x}\exists{y}\psi(x,y)$, where $\psi$ is a $\Delta_{0}$-formula. A realiser of $\phi$ must consist in an OTM that, given a constructible set $X$, computes some constructible set $Y$ and a realiser $r$ for $\psi(x,y)[X/x,Y/y]$. Therefore, it is enough to show that there is a program that for every $X$ computes a set $Y$ such that $\psi(x,y)[X/x,Y/y]$. Split the tape  into $\omega \times \text{On}$ many disjoint portions. Let $(P_i)_i \in \omega$ be a computable enumeration of Turing machine programs. On the $\langle i,\alpha \rangle$-th portion of the tape, run $P_{i}$ on input $\alpha$. In this way, every OTM-computable set will eventually be computed on one of the portions. As the OTM-computable sets coincide with the constructible sets (cf. Koepke \cite{Koepke}), all constructible sets  will eventually be on a portion of the tape. While producing all the constructible sets, our program uses Theorem \ref{Theo:SatRelationDelta0} to look for a $Y$ such that $\psi(x,y)[X/x,Y/y]$ and stops when it finds one. We note that the realiser is uniform since the code of $Y$ only depends on the $\mathrm{OTM}$ enumeration of $\Le$. 
\end{proof}

\begin{proposition}[{cf. \cite[Proposition 9.4.4]{Carl2019}}]
     The axiom schema of separation is not realisable. 
\end{proposition}
\begin{proof}
    Let $\chi(y,y')$ be the formula stating that $y'$ is the power set of $y$. Let $\varphi(x,y)= \exists y' \chi$ be the formula that ignores $x$ and expresses the fact that the power set of $y$ exists. Further assume that $r = \langle c, P \rangle$ realises the corresponding instance of separation:
    $$
        \forall y\forall v \exists z \forall u ((u\in z\rightarrow (u\in v \land \varphi[u/x]))\land ((u\in v \land \varphi[u/x])\rightarrow  u\in z)).
    $$
    Let $Y$ be an arbitrary set, let $a$ be a code for $Y$, $b$ be a code for $\emptyset$ and $c$ a code for $\{ \emptyset \}$. Then $r(a)(c) $ is a realiser of:
    $$
          \exists z \forall u ((u\in z\rightarrow (u\in v \land \varphi[u/x]))\land ((u\in v \land \varphi[u/x])\rightarrow  u\in z))[
        Y/y][\{\emptyset\}/v],
    $$
    i.e. $r(a)(c)(0) = \langle d, s \rangle$, where $d$ is a code for a set and $s$ a realiser.
    Since the formula $\exists y' \chi[Y/y]$ is trivially realised by taking the power set of $Y$ as a parameter, it follows that the set $Z$ coded by $d$ is non-empty and, in fact, $Z = \{ \emptyset \}$. 
    Let $t \Vdash \emptyset \in \{ \emptyset \}$.
    It follows that $s(b)(0)(t)(1) \Vdash \varphi[Y/y]$, i.e. $s(b)(0)(t)(1)$ computes a code for the power set of $X$. This is in contradiction to Lemma \ref{Lemma:Pow}.
\end{proof}


Recall that a choice function $f$ on $X$ satisfies $f(Y) \in Y$ for all $Y \in X$. The \emph{axiom of choice} is the principle that there is a choice function on every set $X$ if $X$ does not contain the empty set. The \emph{well-ordering principle} states that every set is in bijection with an ordinal. We call \emph{weak axiom of choice} the sentence
$\forall X(\forall Y \in X (\exists Z (Z \in Y)) \rightarrow \exists f (  f(Y) \in Y))$. 

\begin{theorem}
    The weak axiom of choice is uniformly realised.
\end{theorem}
\begin{proof}
    For each non-empty $Y \in X$, let $r_Y$ be a uniform realiser of $\exists Z (Z \in Y)$. Then $r_Y = \langle a_Y, s_Y \rangle$, where $a_Y$ is the code of an element $Z_Y$ of $Y$. These realisers can be combined into a function $f$ with $f(Y) = Z_Y$ to obtain a realiser of the weak axiom of choice.
\end{proof}

Despite the fact that the axiom of choice and the weak axiom of choice are classically equivalent, the first one is considerably stronger from the point of view of realisability.

\begin{theorem}\label{Theo:notAC}
    The axiom of regularity, the axiom of choice, and the well-ordering principle are realisable. Moreover, it is consistent that these principles are not uniformly realisable. 
\end{theorem}
\begin{proof}
    Note that the $\in$-minimal elements of a set are computable. For this reason, the axiom of regularity can be realised by returning a code for the first $\in$-minimal element that the machine finds. Similarly, for the axiom of choice, the machine just searches through the code of the input family and for each set in the family output the first element it finds. Finally, the well-ordering principle is realisable because every code of a set induces a well-ordering on the set which can be computed and returned by an OTM. Note that all these algorithms are not uniform because their output crucially depends on the specific codes of the input sets. 

    To prove the second part, note that a uniform realiser of the axiom of regularity can easily be used to uniformly realise the axiom of choice by picking out an ($\in$-minimal) element of each set in the given collection. Similarly, a uniform realiser of the well-ordering principle can be used to uniformly realise the axiom of choice (see \cite[Section 8.6.1]{Carl2019}). 
    
    Finally, note that a uniform realiser of the axiom of choice would provide a class function $F$ such that $F(X)\in X$ for every set $X$. By Rubin \cite[p.201]{rubin1985equivalents} this statement is equivalent to the axiom of global choice in Von Neumann-Bernays-G\"odel set theory. Therefore the consistency of the non-uniformly realisability of the axiom of choice follows by the fact that there are models of $\mathsf{ZFC}$ where the axiom of global choice fails, see, e.g., \cite{JoelMathOverflow}.
    %
\end{proof}

Note that by Theorem \ref{V=L uniform is the same as realisable} if $\Ve=\Le$ then the axiom of regularity, the axiom of choice, and the well-ordering principle are all uniformly realisable.

We end this section with a few comments on the realisability of intuitionistic and constructive set theories. Friedman’s and Myhill's $\mathsf{IZF}$ has the axioms of $\mathsf{ZF}$ but replacing foundation with $\in$-induction and replacement with collection; we saw that not all of the axioms of $\mathrm{IZF}$ are OTM-realisable. Aczel’s $\mathsf{CZF}$ has the axioms of extensionality, pairing, union, empty set, infinity, bounded separation, strong collection, the subset collection scheme and the axiom of set induction. Of these, only the subset collection scheme is not OTM-realised. As we have seen, all axioms of Kripke-Platek set theory are OTM-realisable.

\section{Proof-Theoretic Application: The Admissible Rules of IKP}
\label{Section: Proof-Theoretic Application}

In this section, we will apply our realisability techniques to answer a recent proof-theoretic question of Iemhoff and the third author concerning finitary intuitionistic Kripke-Platek set theory $\IKP$ (see \cite[Question 66]{IemhoffPassmann2020}). The theory $\IKP$ was introduced by Lubarsky \cite{Lubarsky2002} to study admissible sets in the context of intuitionistic set theory. In our context, $\IKP$ is obtained by restricting infinitary Kripke-Platek set theory to the standard finitary language of set theory, and then adding the infinity axiom (see \Cref{Proposition:Infinity realised}).

For the rest of this section, we restrict our attention to the finitary language of set theory. If $T$ is a theory, then a map $\sigma$ from the set of propositional variables to the set of sentences in the language of $T$ is called a $T$-substitution. Given a propositional formula $A$, we write $A^\sigma$ for the formula obtained from $A$ by replacing each propositional letter $p$ with the formula $\sigma(p)$.

\begin{definition}
    A \emph{propositional rule} is an ordered pair $(A,B)$ of formulas in propositional logic, usually written $\frac{A}{B}$. A propositional rule $\frac{A}{B}$ is called \emph{admissible for a theory $T$} if for all substitutions $\sigma$ of propositional letters for sentences in the language of $T$, we have that $T \vdash A^\sigma$ implies that $T \vdash B^\sigma$. A propositional rule $\frac{A}{B}$ is called \emph{derivable for a theory $T$} if $T \vdash A^\sigma \rightarrow B^\sigma$ for all substitutions $\sigma$ of propositional letters for sentences in the language of $T$.
\end{definition}

In contrast to the case of classical propositional logic, it is possible for intuitionistic and intermediate logics to have admissible rules that are not derivable. The structure of the admissible rules of intuitionistic propositional logic has been investigated since the 1970s. Rybakov \cite{Rybakov1997} proved that the set of admissible rules of $\IPC$ is decidable, answering a question of Friedman \cite[Question 40]{Friedman1975}. Visser \cite{Visser1999} later showed that the propositional admissible rules of Heyting Arithmetic are exactly the admissible rules of intuitionistic propositional logic, and Iemhoff \cite{Iemhoff2001} provided an explicit description of the set of admissible rules (a so-called \emph{basis}). 

Recently, Iemhoff and the third author \cite{IemhoffPassmann2020} analysed the structure of propositional and first-order logics of intuitionistic Kripke-Platek set theory $\IKP$, and showed, in particular, that the propositional logic of $\IKP$ is $\IPC$.\footnote{This specific result also follows from an earlier result of the third author \cite{Passmann2020}.} It follows that the set of propositional admissible rules of $\IKP$ is included in the set of propositional admissible rules of $\IPC$. This suggests the question of whether the converse inclusion holds as well (\cite[Question 66]{IemhoffPassmann2020}). We will now show that this is indeed the case, and hence, that the propositional admissible rules of $\IKP$ are exactly the admissible rules of $\IPC$.

For technical purposes, we first need to consider a certain conservative extension of $\IKP$ (and show, of course, that it is indeed conservative). To keep all languages set-sized, we will from now on work with realisability over $\Ve_\kappa$ for some inaccessible cardinal $\kappa$. Hence, all notions of realisability are restricted to $\Ve_\kappa$, and all realisers must be elements of $\Ve_\kappa$ as well. Recall that the \textit{elementary diagram} of $\Ve_\kappa$ contains the statements $c_a \in c_b$ whenever $a \in b$, and $c_a = c_b$ whenever $a = b$.

\begin{definition}
    Let $T$ be an extension of $\IKP$ in the language of set theory. The theory $T^*$ is obtained from $T$ by extending the language of $\IKP$ with constant symbols for every element of $\Ve_\kappa$ and adding the elementary diagram of $\Ve_\kappa$ to $\IKP$.
\end{definition}

\begin{lemma}
    \label{Lemma:IKP^* conservative over IKP}
    Let the theory $T$ be an extension of $\IKP$. Then $T^*$ is conservative (in the constant-free language of set theory) over $T$.
\end{lemma}
\begin{proof}
    Suppose that $T^* \vdash \phi$, where $\phi$ is a sentence in the language of $\IKP$. By compactness we can assume that
    $$
        \bigwedge_{j < m} \psi_j \wedge \bigwedge_{i < n} \chi_i \vdash \phi
    $$
    for $\psi_j \in \IKP$ and $\chi_i$ statements of the form $c_a \in c_b$ or $c_a = c_b$. By the implication rule (see \Cref{Proposition: Implication Rule}), we obtain that
    $$
        \bigwedge_{i < n} \chi_i \vdash \bigwedge_{j < m} \psi_j \rightarrow \phi.
    $$
    Now the formula in the consequent of the sequent is in the language without constants, and the left side contains a finite set of constants $c_i$. We are therefore justified in (repeatedly) applying the existential rule (see \Cref{Proposition: Existential Rule}) to get
    $$
        \exists c_{a_0} \exists c_{a_1} \dots \exists c_{a_k} \bigwedge_{i < n} \chi_i \vdash \bigwedge_{j < m} \psi_j \rightarrow \phi.
    $$
    Now the whole sequent is in the usual language of set theory, without constants. Applying the converse implication rule (see \Cref{Proposition: Implication Rule}), we get
    $$
        \bigwedge_{j < m} \psi_j  \wedge \exists c_{a_0} \exists c_{a_1} \dots \exists c_{a_k} \bigwedge_{i < n} \chi_i \vdash \phi.
    $$
    In this situation it suffices to show that
    $$
        X := \exists c_{a_0} \exists c_{a_1} \dots \exists c_{a_k} \bigwedge_{i < n} \chi_i
    $$
    is a consequence of $\IKP$, and thus a consequence of $T$, to conclude that $\IKP \vdash \phi$. 
    
    To this end, first observe that $X$ describes a finite cycle-free directed graph because $X$ encodes the $\in$-relation between finitely many sets in some $\Ve_\kappa$ (which is, of course, well-founded by the foundation axiom). Hence, to see that $X$ is a theorem of $\IKP$ it suffices to see that every such finite graph can be modelled by using the axioms of empty set, pairing and union. 
\end{proof}

Our next step is to adapt the technique of \emph{glued realisability} to our situation (see van Oosten's survey \cite{vanOosten2002} for the arithmetical version). 

\begin{definition}
    Let $T$ be an OTM-realised theory. We then define the \emph{$T$-realisability relation, $\Vdash_T$} by replacing conditions (iv) and (viii) of \Cref{Definition:OTM-Realisability} with the following clauses:
    \begin{enumerate}
        \item[(iv')] $r \Vdash_T \phi \rightarrow \psi$ if and only if $T^{*} \vdash \phi \rightarrow \psi$ and for every $s \Vdash_T \phi$ we have that $r(s) \Vdash_T \psi$, and,
        \item[(viii')] $r \Vdash_T \forall x \phi$ if and only if $T^{*} \vdash \forall x \phi$ and $r(X) \Vdash_T \phi[X/x]$ for every set $X$.
    \end{enumerate}
\end{definition}

Note that we do not need to redefine condition (viii) in full generality for transfinite quantifiers as we are restricting to the finitary language in this section.

\begin{lemma}
    \label{Lemma: derivability truth IKP^*}
    Let $T$ be an OTM-realised theory extending $\IKP$, and $\phi$ be a formula in the language of $T^*$. Then, $T^* \vdash \phi$ if and only if there is a realiser $r \Vdash_T \phi$.
\end{lemma}
\begin{proof}
    The forward direction is essentially proved in the same way as the fact that realisability is sound with respect to intuitionistic logic (see \Cref{Corollary: Soundness of kappa-first-order logic}) and that all axioms of $T$ are realised, paying attention to the fact that the new clauses (iv') and (viii') do not cause any problems. The backward direction is proved with a straightforward induction on the complexity of the formula $\phi$ (in the extended language of $T^*$), using the definition of $T^*$ and \Cref{Lemma: equality and set-membership are trivial} for the atomic cases. 
\end{proof}

\begin{proposition}
    \label{Proposition: Derivability in IKP equivalent IKP-realisability}
    Let $T$ be an OTM-realised theory extending $\IKP$, and $\phi$ be a formula in the language of $T$. Then, $T \vdash \phi$ if and only if there is a realiser $r \Vdash_T \phi$.
\end{proposition}
\begin{proof}
    This is a direct consequence of the previous \Cref{Lemma:IKP^* conservative over IKP,Lemma: derivability truth IKP^*}. 
\end{proof}

This concludes our preparations and we are now ready to apply this to some proof-theoretic properties.

\begin{theorem}[Disjunction Property]
    \label{Theorem: Disjunction Property}
    Intuitionistic Kripke-Platek Set Theory $\IKP$ has the disjunction property, i.e., if $\IKP \vdash \phi \vee \psi$, then $\IKP \vdash \phi$ or $\IKP \vdash \psi$.
\end{theorem}
\begin{proof}
    This is now a straightforward consequence of our preparations: if $\IKP \vdash \phi \vee \psi$, then $\Vdash_\IKP \phi \vee \psi$ by \Cref{Proposition: Derivability in IKP equivalent IKP-realisability}. By definition, it follows that $\Vdash_\IKP \phi$ or $\Vdash_\IKP \psi$. Applying \Cref{Proposition: Derivability in IKP equivalent IKP-realisability} again yields $\IKP \vdash \phi$ or $\IKP \vdash \psi$.
\end{proof}

The restricted Visser's rules $\set{V_n}_{n < \omega}$ are defined as follows:
\begin{mathpar}
    \inferrule{
            \left( \bigwedge_{i = 1}^n (p_i \rightarrow q_i) \right) \rightarrow (p_{n+1} \vee p_{n+2})
        }
        {
            \bigvee_{j = 1}^{n+2}\left( \bigwedge_{i = 1}^n (p_i \rightarrow q_i) \rightarrow p_j \right) 
        }
\end{mathpar}
Denote by $V_n^a$ the antecedent and by $V_n^c$ the consequent of the rule. We will make use of the following theorem, which is a direct corollary of Iemhoff's results \cite{Iemhoff2005}.

\begin{theorem}[Iemhoff, {\cite[Theorem 3.9, Corollary 3.10]{Iemhoff2005}}]
    \label{Theorem: Iemhoff}
    If the restricted Visser's rules are propositional admissible for a theory $T$ with the disjunction property, then the propositional admissible rules of $T$ are exactly the propositional admissible rules of intuitionistic propositional logic.
\end{theorem}

\begin{theorem}
    The propositional admissible rules of $\IKP$ are exactly the admissible rules of $\IPC$.
\end{theorem}
\begin{proof}
    By \Cref{Theorem: Iemhoff} and the fact that $\IKP$ has the disjunction property (see \Cref{Theorem: Disjunction Property}) it suffices to show that the restricted Visser's rules are propositional admissible. To this end, let $\sigma$ be a substitution. We will write $\phi_i$ for $\sigma(p_i)$, and $\psi_j$ for $\sigma(q_j)$. Now, assume that $\IKP \vdash (V_n^a)^\sigma$, i.e., spelling this out,
    $$
        \IKP \vdash \bigwedge_{i = 1}^n (\phi_i \rightarrow \psi_i) \rightarrow (\phi_{n+1} \vee \phi_{n+2}).
    $$
    Denote the antecedent of $(V_n^a)^\sigma$ by $\delta$. We will now consider the theory $\IKP + \delta$. There are two cases.
    
    In the first case, suppose that there is a realiser $r \Vdash_{\IKP + \delta} \delta$. By assumption, $\IKP \vdash (V_n^a)^\sigma$, and hence, with \Cref{Proposition: Derivability in IKP equivalent IKP-realisability}.(i), it follows that there is a realiser $s \Vdash_{\IKP} (V_n^a)^\sigma$, and thus $s \Vdash_{\IKP + \delta} (V_n^a)^\sigma$. Hence, $s(r) \Vdash_{\IKP + \delta} \phi_{n+1} \vee \phi_{n+2}$. It follows that $\phi_{n+k}$ is $\IKP + D^\sigma$-realised for some $k < 2$. By \Cref{Proposition: Derivability in IKP equivalent IKP-realisability}.(ii), we have $\IKP + \delta \vdash \phi_{n+k}$, and with the deduction theorem and some propositional reasoning, we conclude that $\IKP \vdash (V_n^c)^\sigma$.
    
    In the second case, suppose that $\delta$ is not $\IKP + \delta$-realisable. By the definition of $\IKP + \delta$-realisability, this means that there is some $i$, $1 \leq i \leq n$, such that $\phi_i \rightarrow \psi_i$ is not $\IKP + \delta$-realised. This means that $\IKP + \delta \not \vdash \phi_i \rightarrow \psi_i$, or that for every potential realiser $r$, there is a realiser $s \Vdash \phi_i$ such that $r(s) \not \Vdash \psi_i$. As $\phi_i \rightarrow \psi_i$ is a consequence of $\delta$, it follows, in particular, that there is a realiser $r \Vdash_{\IKP + \delta} \phi_i$. By \Cref{Proposition: Derivability in IKP equivalent IKP-realisability}.(ii), we have that $\IKP + \delta \vdash \phi_i^\sigma$. An application of the deduction theorem yields $\IKP \vdash \delta \rightarrow \phi_i$. In this situation, it is immediate that $\IKP \vdash (V_n^c)^\sigma$.
    
    In conclusion, we have shown that $\IKP \vdash (V_n^a)^\sigma$ implies that $\IKP \vdash (V_n^c)^\sigma$ for every substitution $\sigma$. This shows that Visser's rule $V_n$ is admissible for $\IKP$ for every $n < \omega$.
\end{proof}

The difficulty in generalising this technique to the infinitary case seems to lie in generalising \Cref{Lemma:IKP^* conservative over IKP}.

\begin{question}
    Is it possible to generalise the techniques presented in this chapter to obtain proof-theoretic results for infinitary Kripke-Platek set theory?
\end{question}

\printbibliography

\end{document}